\documentclass[12pt,twoside]{article}
\usepackage{amsmath,amsfonts,amssymb,a4,enumitem,epsfig,array,psfrag}
\usepackage{url}
\usepackage{amsthm}

\usepackage{chngcntr} 
\usepackage{color}

\usepackage{subfig}

\usepackage[section]{placeins}

\usepackage[colorlinks=true]{hyperref}

\graphicspath{{figures/}}

\newtheorem{theo}{Theorem}

\newtheorem{prop}{Proposition}[section]
\newtheorem{coro}[prop]{Corollary}

\newtheorem{lemma}[prop]{Lemma}

\theoremstyle{definition}
\newtheorem{example}[prop]{Example}

\theoremstyle{remark}
\newtheorem{case}{Case}
\counterwithin*{case}{theo}
\counterwithin*{case}{prop}
\counterwithin*{case}{fact}

\newtheorem{rem}[prop]{Remark}

\newcommand{\half}{{\frac{1}{2}}}

\newcommand{\ZZ}{{\mathbb{Z}}}

\newcommand{\RR}{{\mathbb{R}}}

\newcommand{\Tw}{\operatorname{Tw}}

\newcommand{\sgn}{\operatorname{sgn}}

\newcommand{\num}{}

\makeatletter
\newcommand{\subjclass}[2][2010]{%
  \let\@oldtitle\@title%
  \gdef\@title{\@oldtitle\footnotetext{#1 \emph{Mathematics Subject Classification.} #2}}%
}
\newcommand{\keywords}[1]{%
  \let\@@oldtitle\@title%
  \gdef\@title{\@@oldtitle\footnotetext{\emph{Key words and phrases.} #1.}}%
}
\makeatother

\date{November 2, 2014}

\begin{document}
\title{Flip invariance for domino tilings of three-dimensional regions with two floors}
\author{Pedro H. Milet \and Nicolau C. Saldanha}


\maketitle

\begin{abstract}
We investigate tilings of cubiculated regions with two simply connected floors by $2 \times 1 \times 1$ bricks. More precisely, we study the flip connected component for such tilings, and provide an algebraic invariant that ``almost'' characterizes the flip connected components of such regions, in a sense that we discuss in the paper. We also introduce a new local move, the trit, which, together with the flip, connects the space of domino tilings when the two floors are identical.
\end{abstract}

\section{Introduction}


Towards the end of the twentieth century, a lot has been said about tilings of two-dimensional regions by a number of different pieces: in particular, the so-called \emph{domino} and \emph{lozenge} tilings have received a lot of attention.

Kasteleyn \cite{Kasteleyn19611209} showed that the number of domino tilings of a plane region can be calculated via the determinant of a matrix. Conway \cite{conway1990tiling} discovered a technique using groups, that in a number of interesting cases can be used to decide whether a given region can be tesselated by a set of given tiles. Thurston \cite{thurston1990} introduced height functions, and proposed a linear time algorithm for solving the problem of tileability of simply connected plane regions by dominoes. 
In a more probabilistic direction, Jockusch, Propp and Shor \cite{jockusch1998random,cohn1996local} studied random tilings of the so-called Aztec Diamond (introduced in \cite{elkies1992alternating}), and showed the Arctic Circle Theorem. Richard Kenyon and Andrei Okounkov also studied random tilings and, in particular, their relation to Harnack curves \cite{kenyonokounkov2006dimers,kenyonokounkov2006planar}.
The concept of a flip is important in the context of dominoes as well as in that of lozenges. In both cases, two tilings of a simply connected region can always be joined by a sequence of flips (see \cite{saldanhatomei1998overview} for an overview). Also, see \cite{saldanhatomei1995spaces} for considerations on flip connectivity in more general two-dimensional regions.     

However, in comparison, much less is known about tilings of three-dimensional regions. Hammersley \cite{hammersley1966limit} proved results concerning the asymptotic behavior of the number of brick tilings of a $d$-dimensional box when all dimensions go to infinity. In particular, his results imply that the limit $\ell_3 := \lim_{n \to \infty} \frac{\log f(2n)}{(2n)^3}$, where $f(n)$ is the number of tilings of an $n \times n \times n$ box, exists and is finite; as far as we know, its exact value is not yet known, but several upper and lower bounds have been established for $\ell_3$ (see \cite{ciucu1998improved} and \cite{friedland2005theory} for more information on this topic). 
Randall and Yngve \cite{randall2000random} considered tilings of ``Aztec'' octahedral and tetrahedral regions with triangular prisms, which generalize domino tilings to three dimensions; they were able to carry over to this setting many of the interesting properties from two dimensions, e.g. height functions and connectivity by local moves. Linde, Moore and Nordahl \cite{linde2001rhombus} considered families of tilings that generalize rhombus (or lozenge) tilings to $n$ dimensions, for any $n \geq 3$. 
Bodini \cite{bodini2007tiling} considered tileability problems of pyramidal polycubes. Pak and Yang \cite{pak2013complexity} studied the complexity of the problems of tileability and counting for domino tilings in three and higher dimensions, and proved some hardness results in this respect.

\begin{figure}[ht]	
	\centering
    \includegraphics[width=0.4\textwidth]{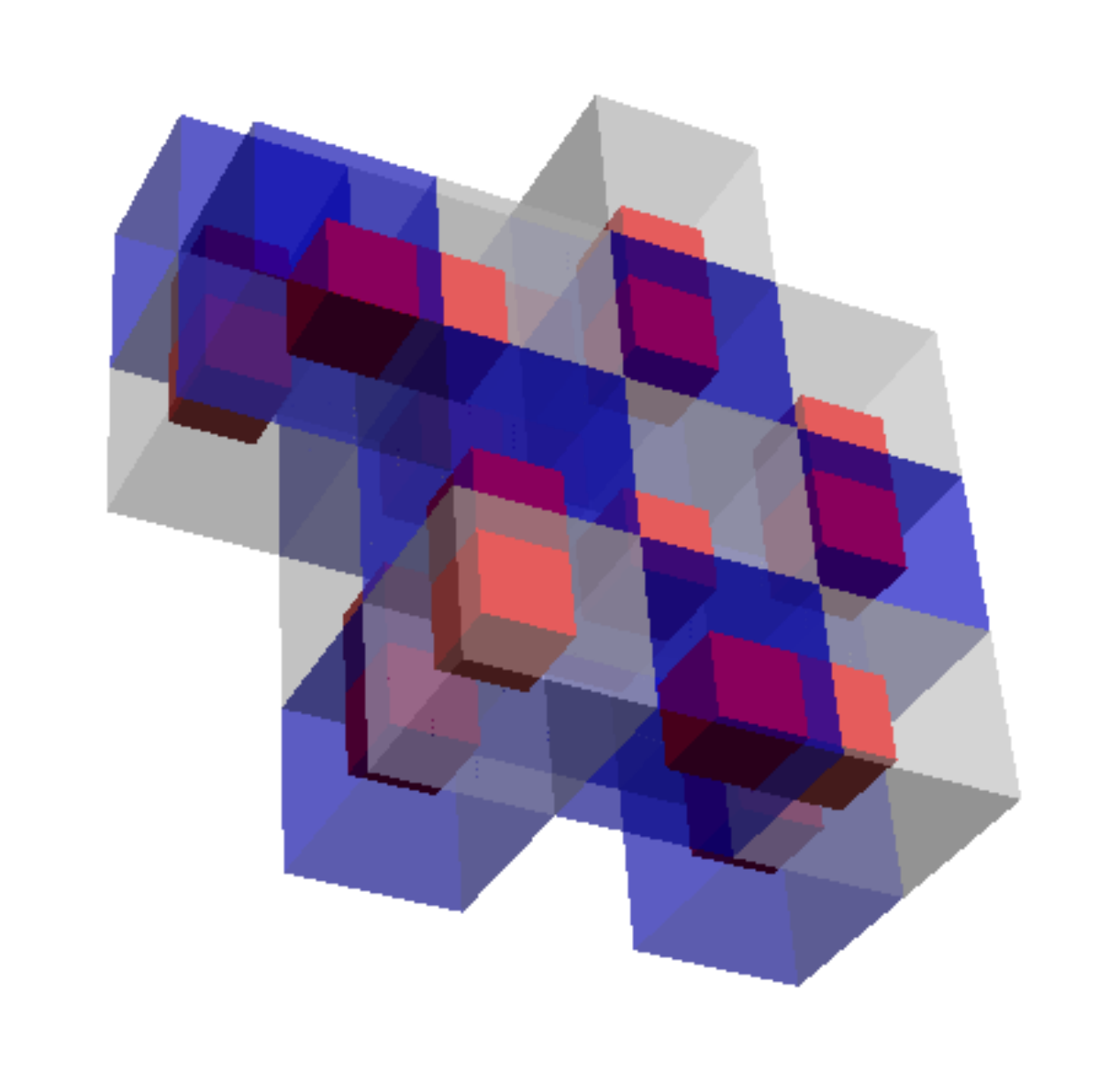}
		\caption{A domino brick tiling of a region with two floors.}
			\label{fig:tilingTwoFloors_example}	
\end{figure}

If $R$ is a cubiculated region in $\RR^3$, a \emph{floor} of $R$ is $R \cap (\RR^2 \times [n,n+1])$, for some $n \in \ZZ$. We say that $R$ is a \emph{two-story region} when it has exactly two non-empty floors, $R \cap (\RR^2 \times [0,1])$ and $R \cap (\RR^2 \times [1,2])$, and both these floors are connected and simply connected. A two-story region is called a \emph{duplex region} if both floors are identical. 

In this paper, we study tilings of two-story regions by \emph{domino brick} pieces (we often call these pieces \emph{dimers}), which are simply $2 \times 1 \times 1$ rectangular cuboids. An example of such a tiling is shown in Figure \ref{fig:tilingTwoFloors_example}. The general case (with an arbitrary number of floors) is discussed in \cite{segundoartigo}.

A quick glance at Figure \ref{fig:tilingTwoFloors_example} shows that while the 3D representation of tilings may be attractive, it is somewhat difficult to work with. Hence, we choose instead to work with a 2D representation that resembles floor plans, which is shown in Figure \ref{fig:notation2Dexample}. The leftmost floor is the top floor. 
The dimers that are parallel to the $x$ or $y$ axis are represented as 2D dominoes, since they are contained in a single floor. The $z-$axis dimers are represented as circles, with the following convention: if the dimer connects a floor with the upper floor, the circle is painted red; otherwise, it is painted white. Thus, for example, in Figure \ref{fig:notation2Dexample}, each of the four white circles on the top floor represents the same dimer as the red circles on the floor directly below it. 

\begin{figure}[ht]
\centering
\def\svgwidth{0.8\columnwidth}
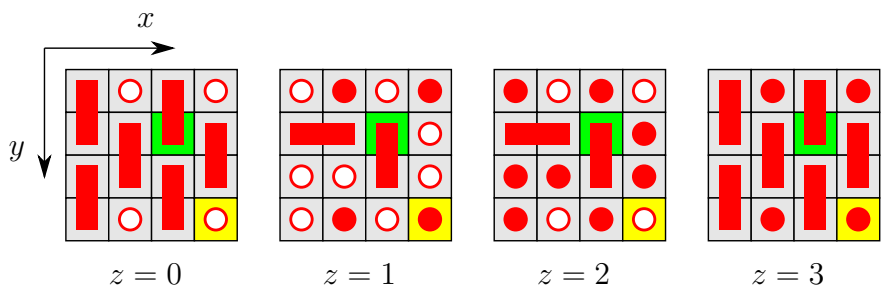
\caption{A tiling of a $4\times 4 \times 4$ box in our notation. The $x$ and $y$ axis are drawn, and the $z$ axis points downwards from the top floor, i.e., the points in the lower floors have higher $z$ coordinates. The squares highlighted in yellow represent cubes whose centers have the same $x$ and $y$ coordinates. Notice the top two yellow cubes are connected by a $z$ dimer, as well as the bottom two. The squares highlighted in green also represent cubes whose center have the same $x$ and $y$ coordinates, but the dimers involving these cubes are not $z$ dimers.}%
\label{fig:notation2Dexample}%
\end{figure}

A key element in our study is the concept of a \emph{flip}, which is a straightforward generalization of the two-dimensional one. We perform a flip on a tiling by removing two (adjacent and parallel) domino bricks and placing them back in the only possible different position. The removed pieces form a $2 \times 2 \times 1$ rectangular cuboid, in one of three possible positions (see Figure \ref{fig:flipExample}). 

In view of the situation with plane tilings, one might expect that the space of domino brick tilings of a simple three-dimensional region, say, an $L \times M \times N$ box, would be connected by flips. This turns out not to be the case, as the spaces of domino brick tilings of even relatively small boxes are already not flip-connected. In this paper, we introduce an algebraic invariant for tilings of two-story regions. This invariant is a polynomial $P_t(q)$ associated with each tiling $t$ such that if two tilings $t_1,t_2$ are in the same flip connected component, then $P_{t_1}(q) \equiv P_{t_2}(q)$. The converse is not necessarily true; however, it is almost true, in a sense that we shall discuss later. 
 
We also introduce a new, genuinely three-dimensional move, which we call a \emph{trit}. A trit is a move that happens within a $2 \times 2 \times 2$ cube with two opposite ``holes'', and it has an orientation. More precisely, we remove three dimers, no two of them parallel, and place them back in the only other possible configuration (see Figure \ref{fig:postrit}). We discuss the effect of a trit in our invariant, and also prove that the space of tilings of a duplex region is connected by flips and trits.   

This paper is structured in the following manner: Section \ref{sec:defsAndNotations} introduces some basic definitions and notations that will be used throughout the paper. In Sections \ref{sec:twoIdenticalFloors} and \ref{sec:generalTwoStory}, we define the polynomial invariant $P_t(q)$ and prove its basic properties. In Section \ref{sec:effectOfTritsOnPt}, we explain how the trit affects the invariant $P_t(q)$ of a tiling. Section \ref{sec:examples} shows examples of regions and their flip connected components, in connection with the invariant. Sections \ref{sec:twoFloorsMoreSpace} and \ref{sec:embedFourFloors} discuss what happens when we embed tilings in a bigger region, in a sense that will be discussed there; in Section \ref{sec:connectedFlipsTrits}, we show that the space of domino tilings of a duplex region is connected by flips and trits. Finally, Section \ref{sec:conclusion} contains a summary of all the results in this paper. It is worth mentioning that, in \cite{segundoartigo}, we investigate tilings of more general regions; there, $P_t(q)$ is no longer defined. However, an important invariant is the \emph{twist} $\Tw(t)$, which, in the case of two-story regions, equals $P_t'(1)$.  

The authors are thankful for the generous support of CNPq, CAPES and FAPERJ. 


\section{Definitions and Notation}
\label{sec:defsAndNotations}

We will sometimes need to refer to numbers of the form $n + \half$, where $n$ is an integer. To avoid heavy notation, for an integer $n$ we will write $n^\sharp := n + \half$. This notation is inspired by music theory, when, say, the note $D\sharp$ is a half tone higher than $D$ in pitch.

By a \emph{cube} we mean a closed unit cube in $\RR^3$ whose vertices lie in $\ZZ^3$. Given $(x,y,z) \in \ZZ^3$, the cube $C\left(x^\sharp , y^\sharp, z^\sharp\right)$ will be $(x,y,z) + [0,1]^3$, i.e., the closed unit cube whose center is $\left(x^\sharp,  y^\sharp, z^\sharp\right)$; it is \emph{white} (resp. \emph{black}) if $x + y + z$ is even (resp. odd).

A \emph{region} is a simply a union of cubes. We may sometimes require the region to satisfy certain hypotheses (for instance, that it is contractible or simply connected), but this will be explicitly stated whenever it is required. A \emph{domino brick} or simply \emph{dimer} is the union of two cubes that share a face. We say two dimers $d_0$ and $d_1$ are \emph{disjoint}, and write $d_0 \cap d_1 = \emptyset$, if their interiors are disjoint (i.e., if they do not contain a common cube). A \emph{tiling} of a region is a covering of a region by disjoint dimers.

A \emph{flip} is a move that takes a tiling $t_0$ into another, by removing two parallel dimers that form a $2 \times 2 \times 1$ ``slab'' and placing them back in the only other position possible. Examples of flips, in the three possible positions, are shown in Figure \ref{fig:flipExample}.

\begin{figure}[ht]%
\centering
\includegraphics[width=0.8\columnwidth]{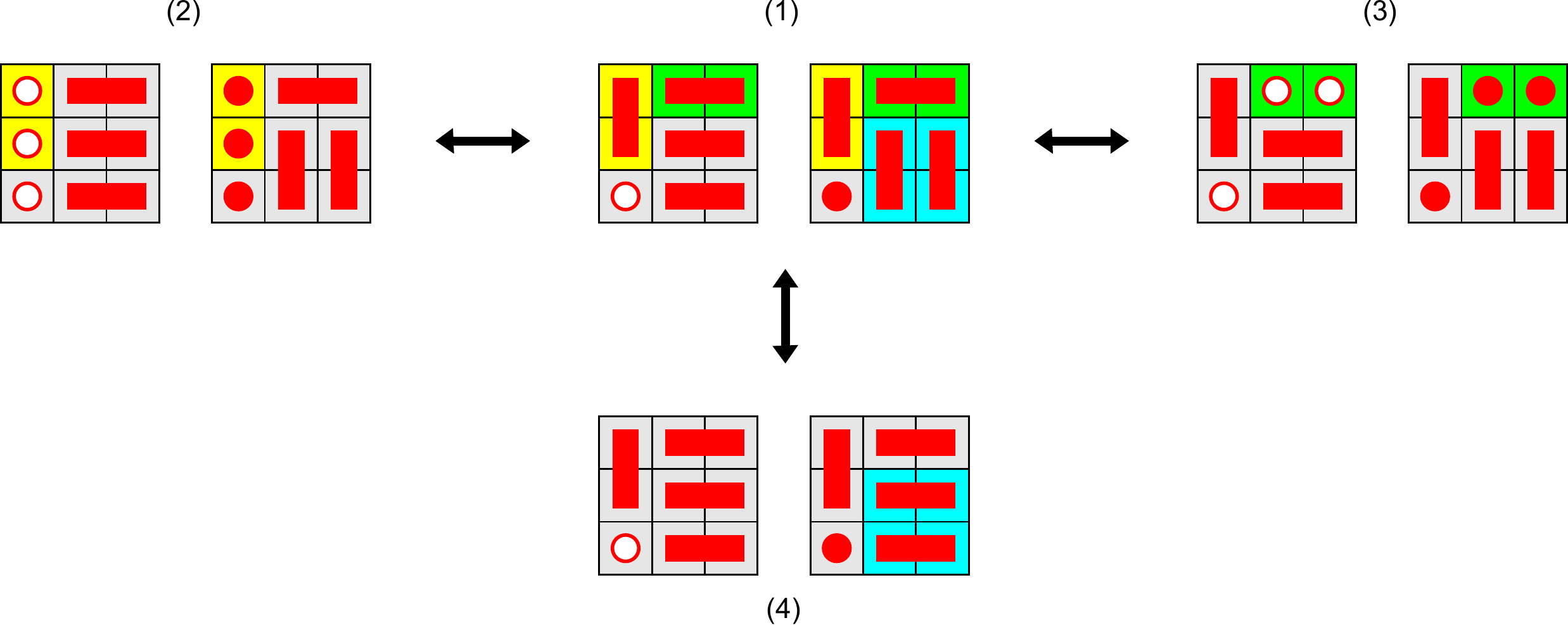}%
\caption{Examples of flips in three different positions. Starting from tiling (1), one can flip to either one of the tilings (2), (3) or (4). The four cubes that form the $2 \times 2 \times 1$ slab where the flip occurs are highlighted in each case.}%
\label{fig:flipExample}%
\end{figure}

We say that a region is \emph{flip connected} if the space of domino brick tilings of the region, thought of as a graph where two tilings are joined by an edge if they differ by a single flip, is connected. In other words, a region is flip connected if, given two tilings of that region, there exists a sequence of flips that takes one tiling to the other. The \emph{flip connected component} of a tiling $t$ is the set of all tilings that can be reached from $t$ after a sequence of flips; hence, a region $R$ is flip connected if and only all tilings of $R$ lie in the same flip connected component. 

Since plane domino tilings may be seen as a special case of domino brick tilings with only one floor, it follows, for instance, that $L \times M \times 1$ boxes are always flip connected. Additionally, it is very easy to show, by induction, that $L \times 2 \times 2$ boxes are flip connected, and we invite the reader to prove this as an exercise.

However, no other box is flip connected (see \cite{segundoartigo}). For instance, the $3 \times 3 \times 2$ box has 229 tilings, two of which have no flip positions. These are shown in Figures \ref{fig:tilings332_noflips1} and \ref{fig:tilings332_noflips2}. This box has, in fact, three flip connected components (two of which contain just one tiling). In the next section, we will look at an algebraic flip-invariant for two-story regions (like the box under consideration).

\begin{figure}[ht]
\centering
\subfloat[][]{\includegraphics[scale=0.5]{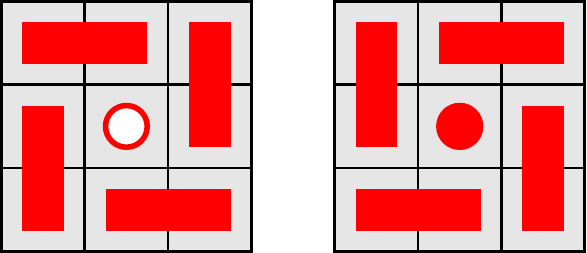}\label{fig:tilings332_noflips1}} \qquad \qquad
\subfloat[][]{\includegraphics[scale=0.5]{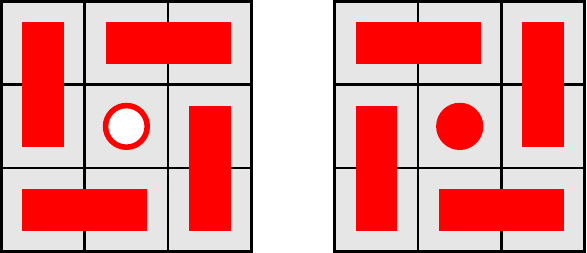}\label{fig:tilings332_noflips2}}

\caption{The two tilings of a $3 \times 3 \times 2$ box that have no flip positions.}

\end{figure}

The fact that even relatively small boxes are not flip connected invited us to consider a new move, which we called the \emph{trit}. Just as the flip takes place inside a $2 \times 2 \times 1$ rectangular cuboid, the trit takes place inside a $2 \times 2 \times 2$ cube where three dimers occur in the three possible positions ($x$, $y$ and $z$). We thus necessarily have some rotation of Figure \ref{fig:postrit}.

\begin{figure}[ht]
\centering
\includegraphics[width=0.6\columnwidth]{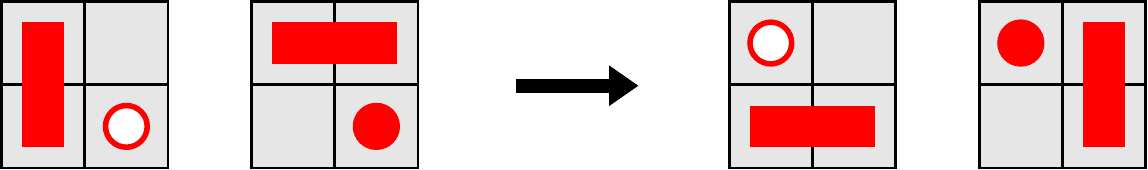}%
\caption{The anatomy of a positive trit (from left to right). The trit that takes the right drawing to the left one is a negative trit. The empty squares may represent either dimers that are not contained in the $2 \times 2 \times 2$ cube or cubes that are not contained in the region (for instance, if the region happens not to be a box).}%
\label{fig:postrit}%
\end{figure}

The trit that takes the drawing at the left of Figure \ref{fig:postrit} to the drawing at the right of it is a \emph{positive trit}. The reverse move is a \emph{negative trit}. Figure \ref{fig:negtrit_example} shows an example of a trit in a $3 \times 3 \times 2$ box.

\begin{figure}[ht]
\centering
\includegraphics[width=0.6\columnwidth]{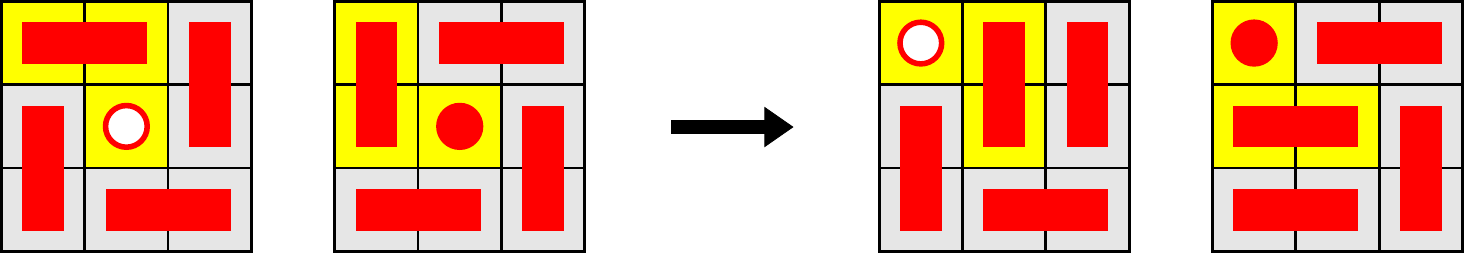}%
\caption{An example of a negative trit. The affected cubes are highlighted in yellow.}%
\label{fig:negtrit_example}%
\end{figure}

\section{Duplex regions}
\label{sec:twoIdenticalFloors}
Let $D \subset \RR^2$ be a quadriculated simply connected plane region, and let $R = D \times [0,2] \subset \RR^3$ be a duplex region, i.e., a two-story region with two identical floors. Our goal throughout this section is to associate to each tiling $t$ of $R$ a polynomial $P_t \in \ZZ[q,q^{-1}]$, which will be the same for tilings in the same flip connected component.

Consider the two floors of a tiling $t$ of $R$. Clearly, $z$ dimers (dimers that are parallel to the $z$ axis) occur in the same positions in both floors. Hence, if we project all the non-$z$ dimers of the bottom floor into the top floor, we will see two plane tilings of $D$ with ``stones'' (which occupy exactly one square) happening precisely in $z$ dimer positions, as shown in Figure \ref{fig:tiling742_invariant}. Since the $z$ dimer positions will play a key role in the associated drawing, we will call them \emph{jewels}. A white (resp. black) jewel is a jewel that happens in a white (resp. black) square.

\begin{figure}[ht]
\centering
\includegraphics[width=0.75\columnwidth]{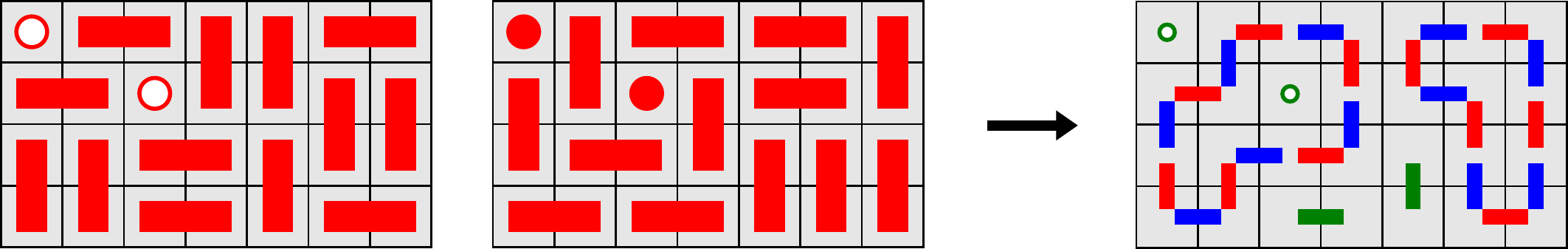}%
\caption{A tiling of a $7 \times 4 \times 2$ box and its associated drawing. The associated drawing has two jewels and four cycles, two of which are trivial ones. The jewels have opposite color; assuming that the topmost leftmost square is white, both cycles spin counterclockwise. }%
\label{fig:tiling742_invariant}%
\end{figure} 

Thus, the resulting drawing can be seen as a set of disjoint plane cycles (the dimers that were projected from the bottom floor are oriented from black to white) together with jewels. A cycle is called \textit{trivial} if it has length $2$. In order to construct $P_t$, consider a jewel $j$, and let $k_t(j)$ be the number of counter-clockwise cycles enclosing the jewel minus the number of clockwise cycles enclosing the jewel (in other words, $k_t(j)$ is the sum of the winding numbers of all the cycles with respect to $j$). Then 
$$P_t(q) = \sum_{j \text{ black}}q^{k_t(j)} - \sum_{j \text{ white}}q^{k_t(j)}.$$
For instance, for the tiling $t$ in Figure \ref{fig:tiling742_invariant}, if the topmost, leftmost square is white, $P_t(q) = q - 1$. Also, define the twist of a tiling to be $\Tw(t) = P_t'(1).$

We now show that $P_t$ is, in fact, flip invariant.

\begin{prop}
\label{prop:twoFloorFlipInvariant}
Let $R$ be a duplex region, and let $t_0$ be a domino brick tiling of $R$. Suppose $t_1$ is obtained from $t_0$ by performing a single flip on $t_0$. Then $P_{t_1} = P_{t_0}$. 
\end{prop}
\begin{proof}
Let us consider the tiling $t_0$ and its associated drawing as a plane tiling with jewels; we want to see how this drawing is altered by a single flip. We will split the proof into cases, and the reader may find it easier to follow by looking at Figure \ref{fig:flipCasesTwoFloors}.

\begin{figure}[ht]
\centering
\subfloat[Case \ref{case:flipzdimer}: $P_t(q) = 1 - q^{-2}$ in both tilings]{\includegraphics[scale=0.4]{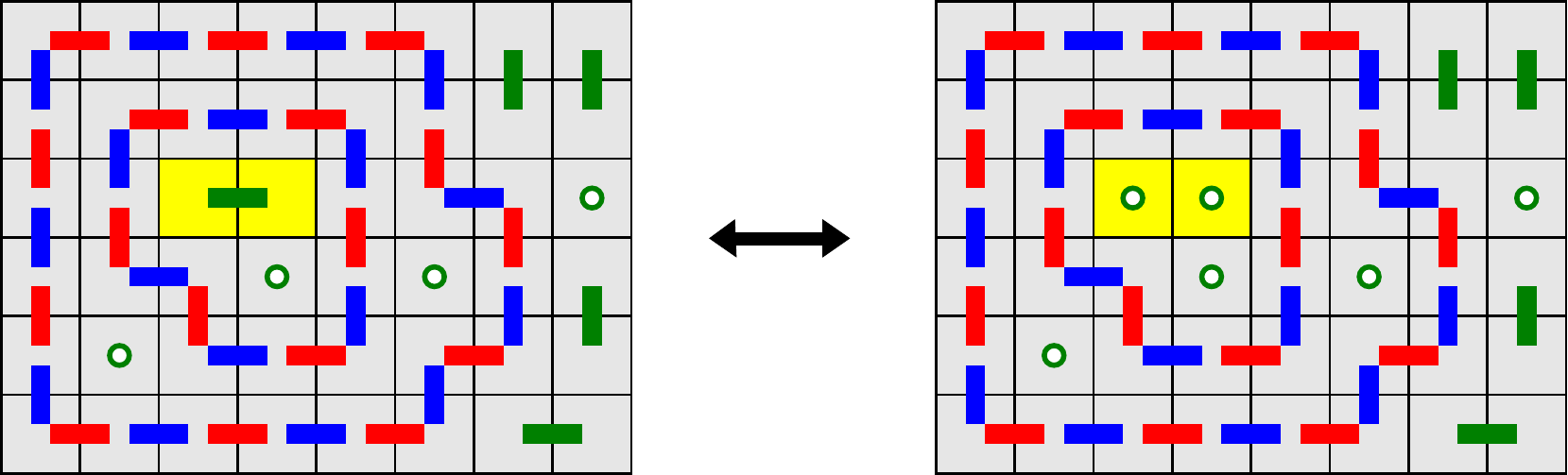}\label{fig:flipcase1}} \\
\subfloat[Case \ref{enum:sameorient}: $P_t(q) = q - 1$ in both tilings.]{\includegraphics[scale=0.4]{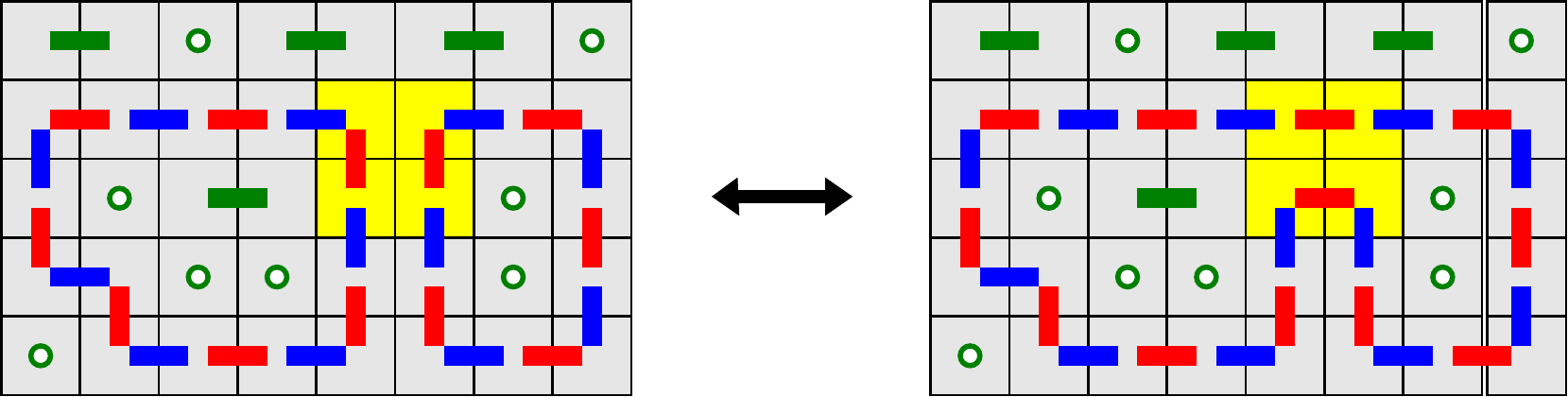}\label{fig:flipcase2_1}} \\
\subfloat[Case \ref{enum:differentorient}: $P_t(q) = q - 1$ in both tilings.]{\includegraphics[scale=0.4]{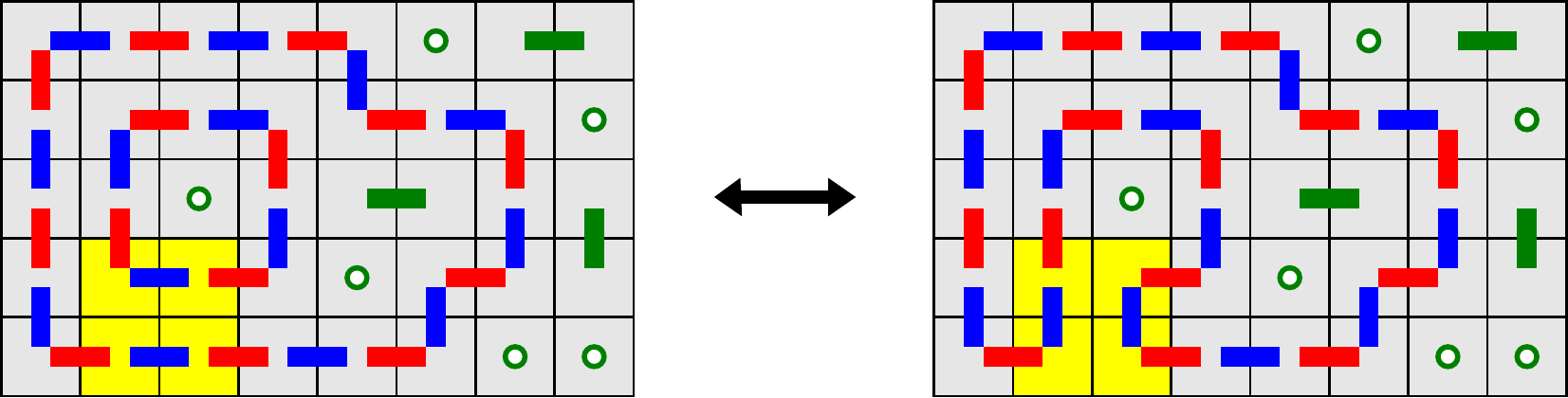}\label{fig:flipcase2_2}}
\caption{Examples illustrating the effects of flips in each of the cases. The flip positions are highlighted in yellow.}
\label{fig:flipCasesTwoFloors}
\end{figure}

\begin{case} 
\label{case:flipzdimer}
A flip that takes two non-$z$ dimers that are in the same position in both floors to two adjacent $z$ dimers (or the reverse of this flip)
\end{case}
The non-$z$ dimers have no contribution to $P_{t_0}$. On the other hand, the two $z$ dimers that appear after the flip have opposite colors (since they are adjacent), and are enclosed by exactly the same cycles. Hence, their contributions to $P_{t_1}$ cancel out, and thus $P_{t_0} = P_{t_1}$ in this case.

\begin{case}
\label{case:fliponefloor}
A flip that is completely contained in one of the floors.
\end{case}
If we look at the effect of such a flip in the associated drawing, two things can happen:
\begin{enumerate}[label=2\alph*., ref=2\alph*]
	\item\label{enum:sameorient} It connects two cycles that are not enclosed in one another and have the same orientation, and creates one larger cycle with the same orientation as the original ones, or it is the reverse of such a move;
	\item\label{enum:differentorient} It connects two cycles of opposite orientation such that one cycle is enclosed by the other (or it is the reverse of such a flip). The new cycle has the same orientation as the outer cycle.
\end{enumerate} 

In case \ref{enum:sameorient}, a jewel is enclosed by the new larger cycle if and only if it is enclosed by (exactly) one of the two original cycles. Hence, its contribution is the same in both $P_{t_0}$ and $P_{t_1}$.

In case \ref{enum:differentorient}, a jewel is enclosed by the new cycle if and only if it is enclosed by the outer cycle and not enclosed by the inner one. If it is enclosed by the new cycle, its contribution is the same in $P_{t_0}$ and $P_{t_1}$, because the new cycle has the same orientation as the outer one. On the other hand, if a jewel $j$ is enclosed by both cycles, their contributions to $k_t(j)$ cancel out, hence the jewel's contribution is also the same in $P_{t_0}$ and $P_{t_1}$. 

Hence, $P_{t_0} = P_{t_1}$ whenever $t_0$ and $t_1$ differ by a single flip.
\end{proof}
Notice that in cases \ref{enum:sameorient} and \ref{enum:differentorient} in the proof, one or both of the cycles involved may be trivial cycles. However, this does not change the analysis.

\section{General two-story regions}
\label{sec:generalTwoStory}

Figure \ref{fig:unequalFloors} shows an example of a region with two unequal simply connected floors. If we try to look at the associated drawing, in the same sense as in Section \ref{sec:twoIdenticalFloors}, we will end up with a set of disjoint (simple) paths. Some of them are cycles, while others have loose ends: an example is showed in the left of Figure \ref{fig:unequalFloors} (in that case, all paths have loose ends).

\begin{figure}[ht]%
\centering
\includegraphics[width=0.5\columnwidth]{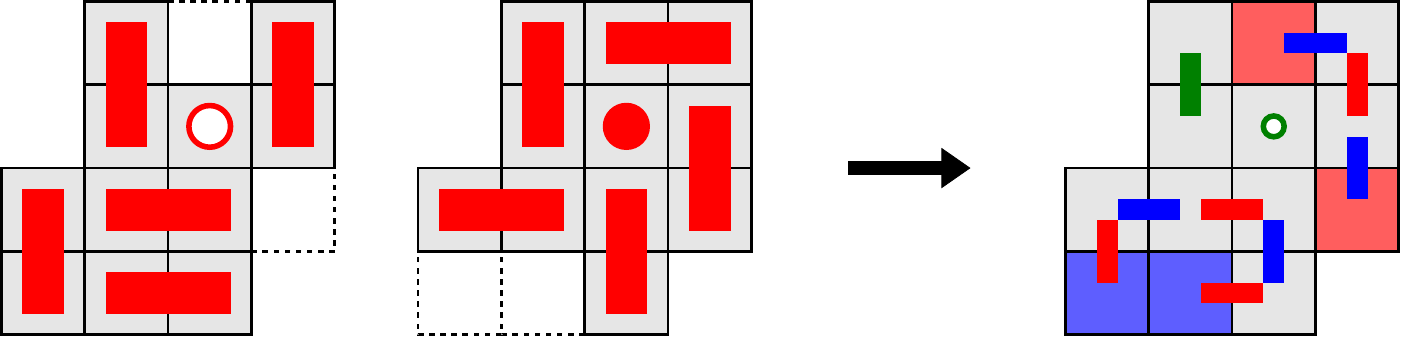}%
\caption{An example of tiled region with two simply connected but unequal floors.}%
\label{fig:unequalFloors}%
\end{figure}

In Figure \ref{fig:unequalFloors}, there are four highlighted squares, which represent cubes in the symmetric difference of the floors. These highlighted squares are called \emph{holes}. The holes highlighted in blue in Figure \ref{fig:unequalFloors} stem from cubes $C(x^\sharp, y^\sharp, 0^\sharp)$ in the top floor such that $C(x^\sharp, y^\sharp, 1^\sharp) \notin R$, whereas the ones highlighted in red stem from cubes $C(x^\sharp, y^\sharp, 1^\sharp)$ in the bottom floor such that $C(x^\sharp, y^\sharp, 0^\sharp) \notin R$. Notice also that every white (resp. black) hole (regardless of which floor it is in) creates a loose end where a dimer is oriented in such a way that it is entering (resp. leaving) the square, which we call a \emph{source} (resp \emph{sink}): the names source and sink do not refer to the dimers, but instead refer to the ghost curves, which are defined below.  Also, sources and sinks do not depend on the specific tiling and, since the number of black holes must equal the number of white holes, the number of sources always equals the number of sinks in a tileable region.

If we connect pairs of loose ends in such a way that a source is always connected to a sink by a curve oriented from the source to the sink, it follows that we now actually see a set of (not necessarily disjoint) cycles, as shown in Figure \ref{fig:unequalFloorsJoined}. Notice that, since the floors are simply connected, we can always connect any source to any sink in the associated drawing via a path that never touches a closed square that is common to both floors, i.e., which may only cross holes (notice that a hole may never contain a jewel). Hence, a \textit{ghost curve} is defined as a curve that connects a source to a sink and which never touches the closure of a square that is common to both floors. 

Hence, for a tiling of a two-story region with unequal floors, the associated drawing is the ``usual'' associated drawing (from Section \ref{sec:twoIdenticalFloors}) together with a set of ghost curves such that each source and each sink is in exactly one ghost curve.   

\begin{figure}[ht]%
\centering
\includegraphics[width=0.7\columnwidth]{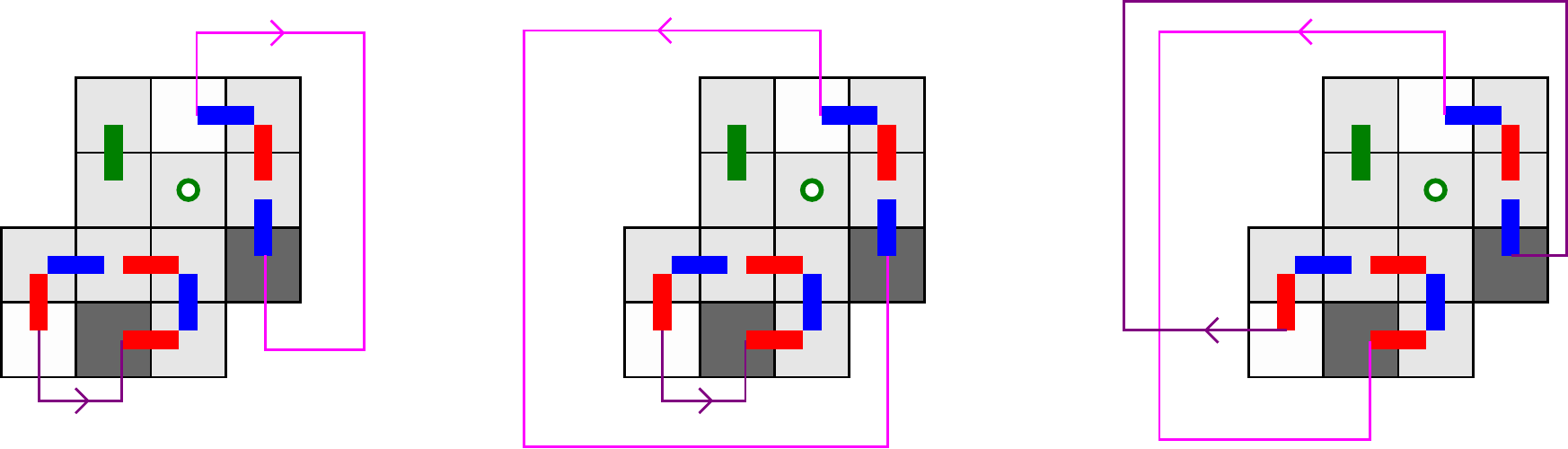}%
\caption{Three different ways to join sources and sinks. The sources are highlighted in very light grey (almost white), while the sinks are highlighted in dark grey. The invariants $P_t(q)$ for each case, from left to right, are $1$, $q$ and $1$.}%
\label{fig:unequalFloorsJoined}%
\end{figure}

Therefore, since for each tiling $t$ the associated drawing is a set of cycles and jewels, we may define $P_t(q)$ in essentially the same way as we did for duplex regions: if for a jewel $j$ we let $k_t(j)$ be the sum of the winding numbers of the cycles with respect to that jewel, we set $P_t(q) = \sum_{j \text{ black}}q^{k_t(j)} - \sum_{j \text{ white}}q^{k_t(j)}.$ The only difference with respect to Section \ref{sec:twoIdenticalFloors} is that the winding number of a cycle with respect to a jewel is no longer necessarily $1$ or $-1$, but can be any integer (see Figure \ref{fig:unequalTwoFloor2}).

\begin{prop} \label{prop:twoStoryFlipInvariant} Let $R$ be a two-story region, and suppose $t_1$ is obtained from a tiling $t_0$ of $R$ after a single flip. Then $P_{t_0} = P_{t_1}$.
\end{prop}
\begin{proof}
The proof is basically the same as that of Proposition \ref{prop:twoFloorFlipInvariant}. In fact, for Case \ref{case:fliponefloor} (flips contained in one floor) in that proof, literally nothing changes, whereas Case \ref{case:flipzdimer} (flips involving jewels) can only happen in a pair of squares that are common to both floors. Since the ghost curves must never touch such squares, it follows that the pair of adjacent jewels that form a flip position have the same $k_t(j)$, hence their contributions cancel out.
\end{proof}

\begin{figure}[ht]%
\centering
\includegraphics[width=0.9\columnwidth]{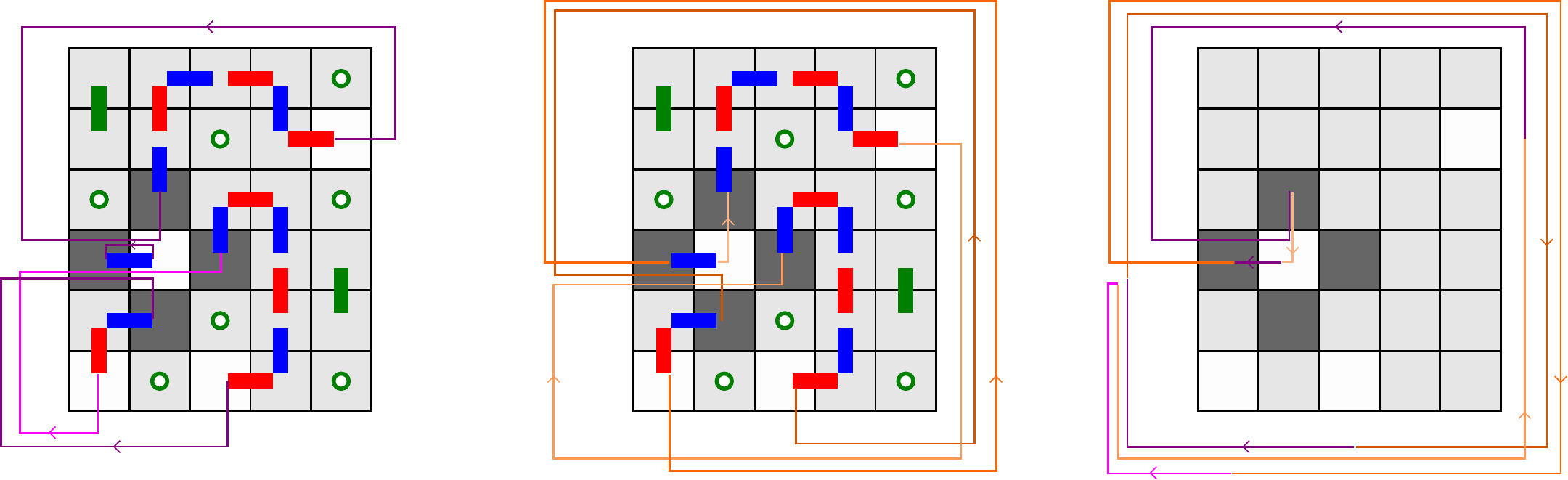}%
\caption{Two different ways to connect sources and sinks. The last diagram illustrates that the difference (first minus second) between these two connections forms a set of cycles, each of which has the same winding number with respect to each square where it is possible to have a jewel. In this case, this set has two cycles: one with winding number $0$ and the other with winding number $-1$ (with respect to every square that is common to both floors). Notice that the invariant is $-2q + 1 - 2q^{-1}$ in the first diagram and $-2q^2 + q - 2$ in the second, that is, the first invariant is indeed $q^{-1}$ times the second.}%
\label{fig:unequalTwoFloor2}%
\end{figure}

A very natural question at this point is the following: how does the choice of ``connections'' affect $P_t$? It turns out that this question has a very simple answer: if $P_{t,1}$ is the invariant associated with one choice of connection and $P_{t,2}$ is associated with another, then there exists $k \in \ZZ$ such that for every tiling $t$, $P_{t,1}(q) = q^k P_{t,2}(q)$. 

To see this, fix a tiling $t$. We want to look at the contributions of a jewel to $P_t$ for two given choices of source-sink connections. Since the exponent of the contribution of a jewel is the sum of the winding numbers of all the cycles with respect to it, it follows that the difference of exponents between two choices of connections is the sum of winding numbers of the cycles formed by putting the ghost curves from both source-sink connections together in the same picture, as illustrated in Figure \ref{fig:unequalTwoFloor2}. Now this sum of winding numbers is the same for every jewel, because the set of ghost curves must enclose every jewel in the same way. Hence, the effect of changing the connection is multiplying the contribution of each jewel by the same integer power of $q$.

Therefore, the invariant for the general case with two simply connected floors is uniquely defined up to multiplication by an integer power of $q$, $q^k$. As a consequence, the twist $\Tw(t) = P_t'(1)$ is defined up to the additive constant $k$.

\section{The effect of trits on $P_t$}     
\label{sec:effectOfTritsOnPt}

\begin{figure}[ht]
\centering
\includegraphics[width=0.4\columnwidth]{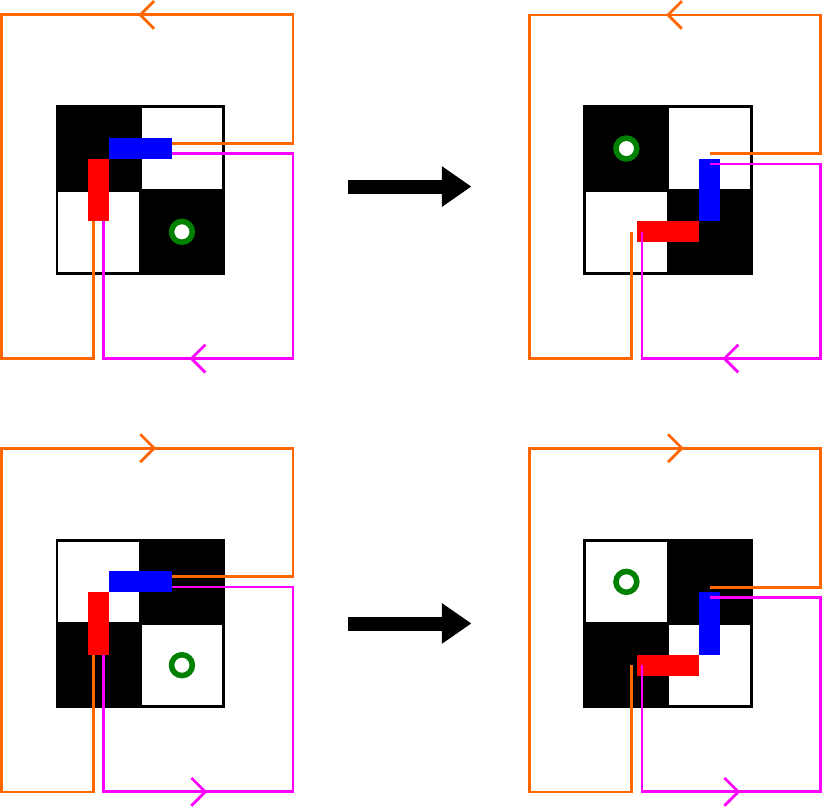}%
\caption{Schematic drawing of the effect of positive trits, with the magenta and orange lines indicating (schematically) the two possible relative positions of the cycle $\gamma$ altered by the trit (the magenta and orange segments represent cycles that may have ghostly parts or not). It is clear that, in either case, the contribution of the portrayed jewel changes from $q^k$ to $q^{k+1}$ if it is black, and from $-q^k$ to $-q^{k-1}$ if it is white.}%
\label{fig:postrit_contrib}%
\end{figure}

We now turn our attention to the effect of a trit on $P_t$. By looking at Figure \ref{fig:postrit_contrib}, we readily observe that a trit affects only the contribution of the jewel $j$ that takes part in the trit (a trit always changes the position of precisely one jewel, since it always involves exactly one $z$ dimer). Furthermore, it either pulls a jewel out of a cycle, or pushes it into a cycle, but the jewel maintains its color. Hence, if $\pm q^k$ is the contribution of $j$ to $P_t$, then after a trit involving $j$ its contribution becomes $\pm q^{k + 1}$ or $\pm q^{k - 1}$.

A more careful analysis, however, as portrayed in Figure \ref{fig:postrit_contrib}, shows that a positive trit involving a black jewel always changes its contribution from $q^k$ to $q^{k+1}$, and a positive trit involving a white jewel changes $-q^k$ to $-q^{k-1}$. Hence, we have proven the following:

\begin{prop} \label{prop:posTrit} Let $t_0$ and $t_1$ be two tilings of a region $R$ which has two simply connected floors, and suppose $t_1$ is reached from $t_0$ after a single positive trit. Then, for some $k \in \ZZ$, 
\begin{equation}
P_{t_1}(q) - P_{t_0}(q) =  q^k(q - 1).
\label{eq:effectOnPt}
\end{equation}
\end{prop}
A closer look at equation \ref{eq:effectOnPt} shows that $\Tw(t_1) - \Tw(t_0) = P_{t_1}'(1) - P_{t_0}'(1) = 1$ whenever $t_1$ is reached from $t_0$ after a single positive trit. This gives the following easy corollary: 

\begin{coro} \label{coro:tritstwofloors} Let $t_0$ and $t_1$ be two tilings of a region $R$ with two simply connected floors, and suppose we can reach $t_1$ from $t_0$ after a sequence $S$ of flips and trits. Then $$\#(\text{positive trits in } S) - \#(\text{negative trits in } S) = \Tw(t_1) - \Tw(t_0).$$  
\end{coro}

\section{Examples}
\label{sec:examples}
For the examples below, we wrote programs in the C$^\sharp$ language.
\begin{example}[The $7 \times 3 \times 2$ box] \label{example:732box} The $7 \times 3 \times 2$ box has a total of $880163$ tilings, and thirteen flip connected components. Table \ref{table:CC732} contains information about the invariants of these connected components. 

\begin{table}[ht]
\begin{tabular}{|c|c|c|c|c|}
\hline
 \parbox[c]{2cm}{\centering Connected \\Component} & \parbox[c]{2cm}{\centering Number of \\tilings} & Tiling &$P_t(q)$&$\Tw(t)$\\ \hline
 0 & 856617 & \includegraphics[scale=0.3]{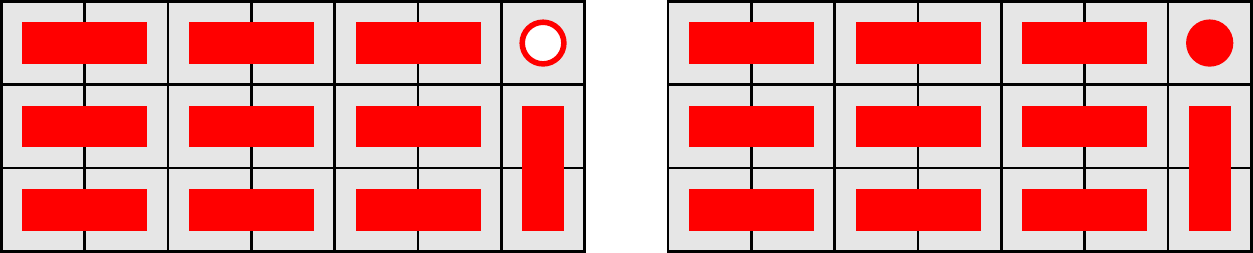} & $-1$ & $0$  \\ \hline
 1 & 8182 & \includegraphics[scale=0.3]{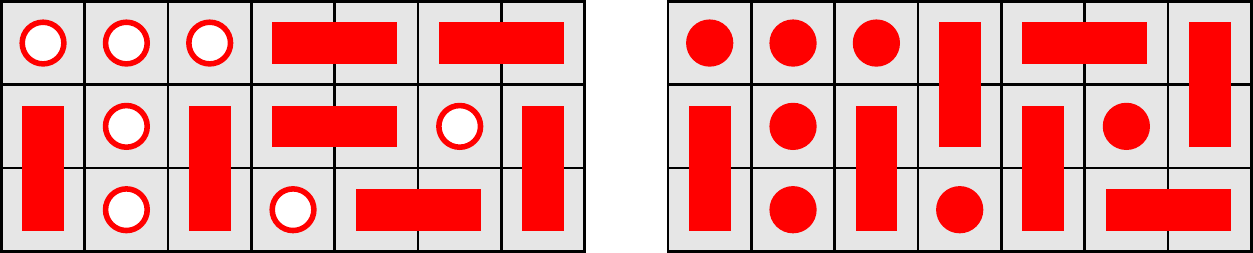} & $-q$ & $-1$  \\ \hline
 2 & 8182 & \includegraphics[scale=0.3]{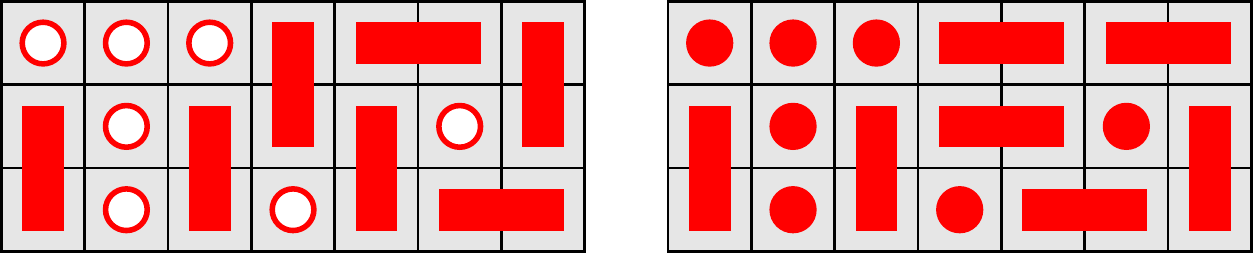} & $-q^{-1}$ & $1$  \\ \hline
 3 & 3565 & \includegraphics[scale=0.3]{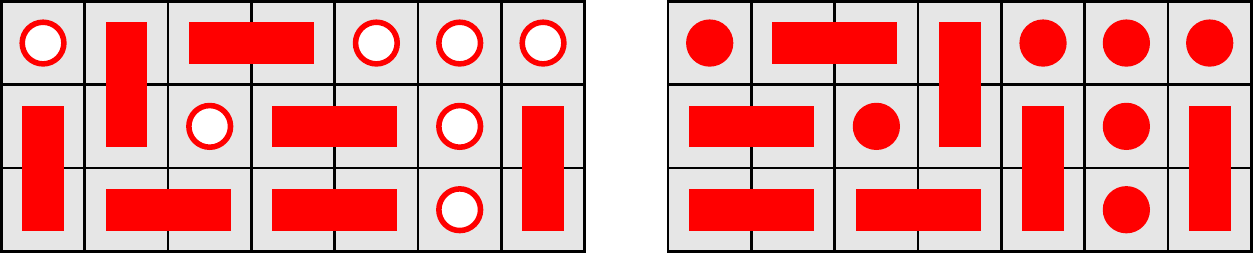} & $-2 + q^{-1}$ & $-1$  \\ \hline
 4 & 3565 & \includegraphics[scale=0.3]{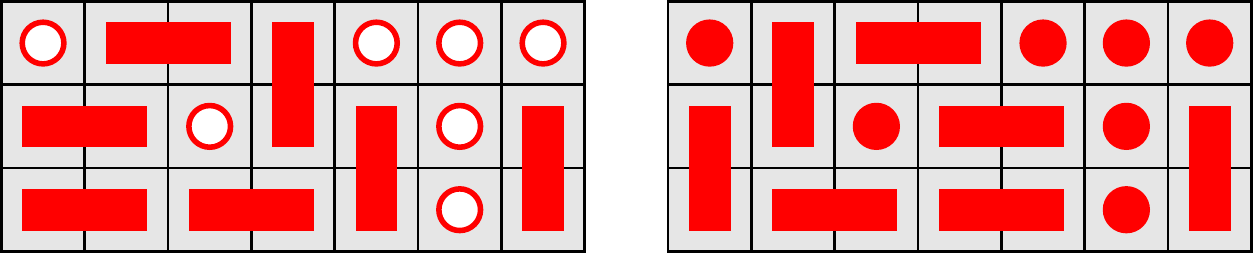} & $q - 2$ & $1$  \\ \hline
 5 & 9 & \includegraphics[scale=0.3]{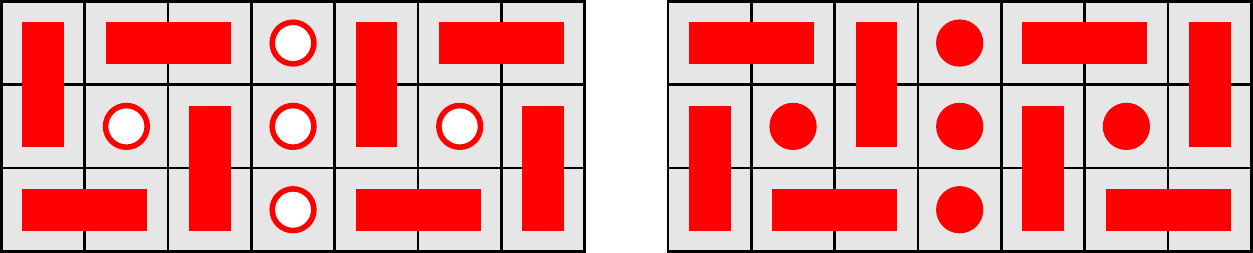} & $-2q + 1$ & $-2$  \\ \hline
 6 & 9 & \includegraphics[scale=0.3]{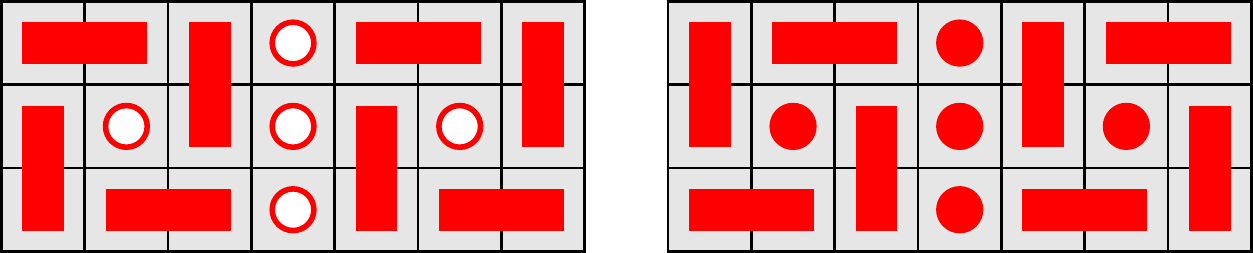} & $1 - 2q^{-1}$ & $2$   \\ \hline 
 7 & 7 & \includegraphics[scale=0.3]{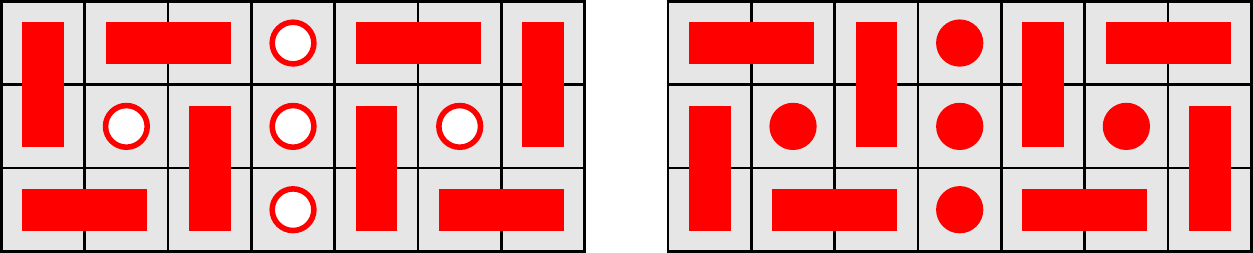} & $-q + 1 - q^{-1}$ & $0$  \\ \hline 
 8 & 7 & \includegraphics[scale=0.3]{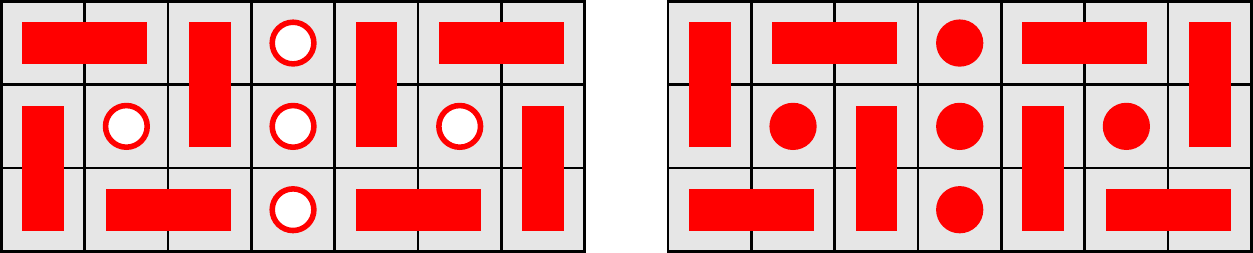} & $-q + 1 - q^{-1}$ & $0$  \\ \hline
 9 & 5 & \includegraphics[scale=0.3]{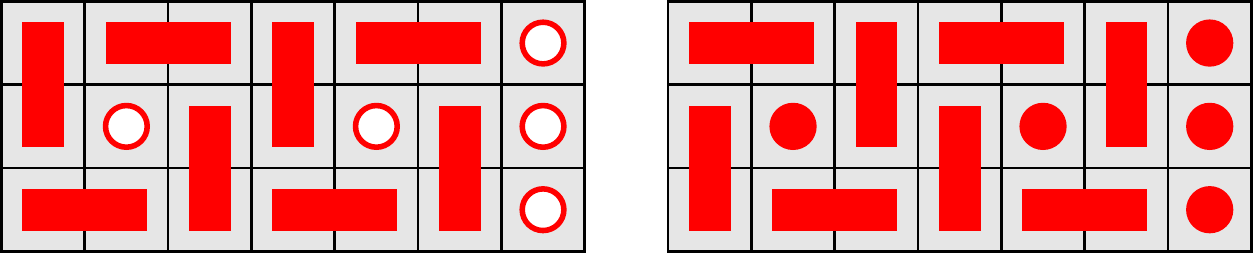} & $-q - 1 + q^{-1}$ & $-2$  \\ \hline
 10 & 5 & \includegraphics[scale=0.3]{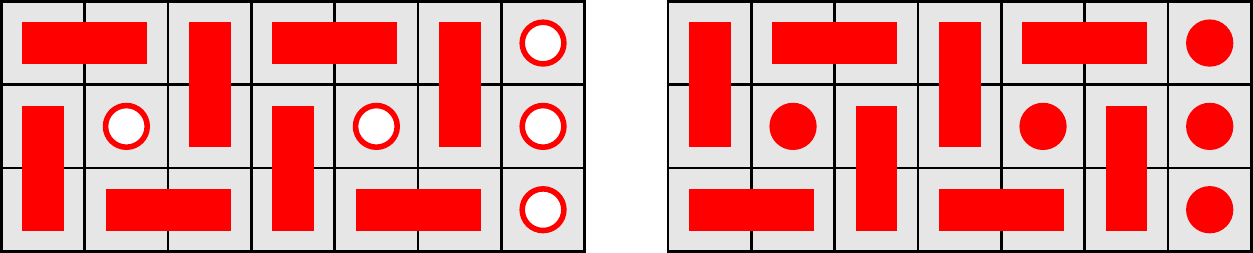} & $q - 1 - q^{-1}$ & $2$  \\ \hline
 11 & 5 & \includegraphics[scale=0.3]{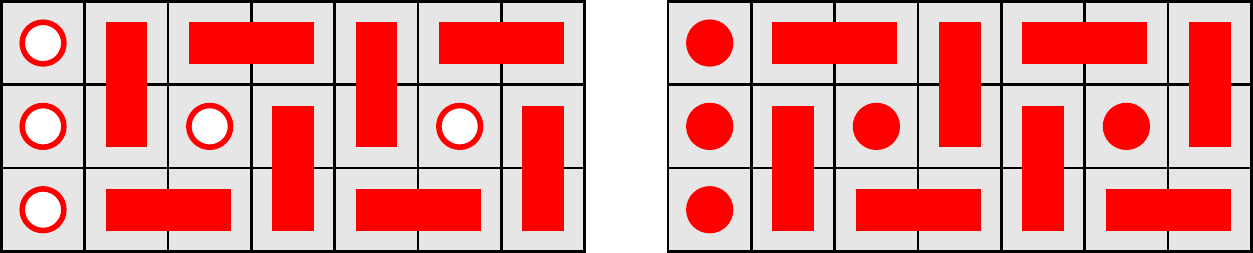} & $-q - 1 + q^{-1}$ & $-2$  \\ \hline
 12 & 5 & \includegraphics[scale=0.3]{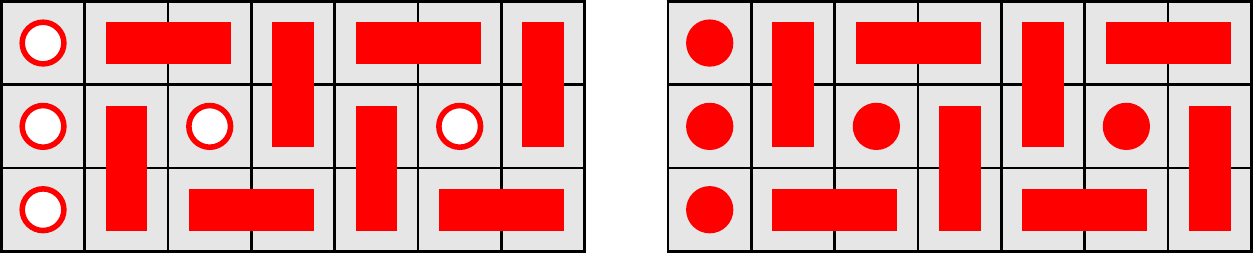} & $q - 1 - q^{-1}$ & $2$   \\ \hline

\end{tabular}
\caption{Flip connected components of a $7 \times 3 \times 2$ box.}
\label{table:CC732}
\end{table}

We readily notice that the invariant does a good job of separating flip connected components, albeit not a perfect one: the pairs of connected components 7 and 8, 9 and 11, and 10 and 12 have the same $P_t(q)$. 

Figure \ref{fig:box732_CCdiagram} shows a diagram of the flip connected components, arranged by their twists. We also notice that we can always reach a tiling from any other via a sequence of flips and trits.

\begin{figure}[ht]
\centering
\includegraphics[width=0.5\columnwidth]{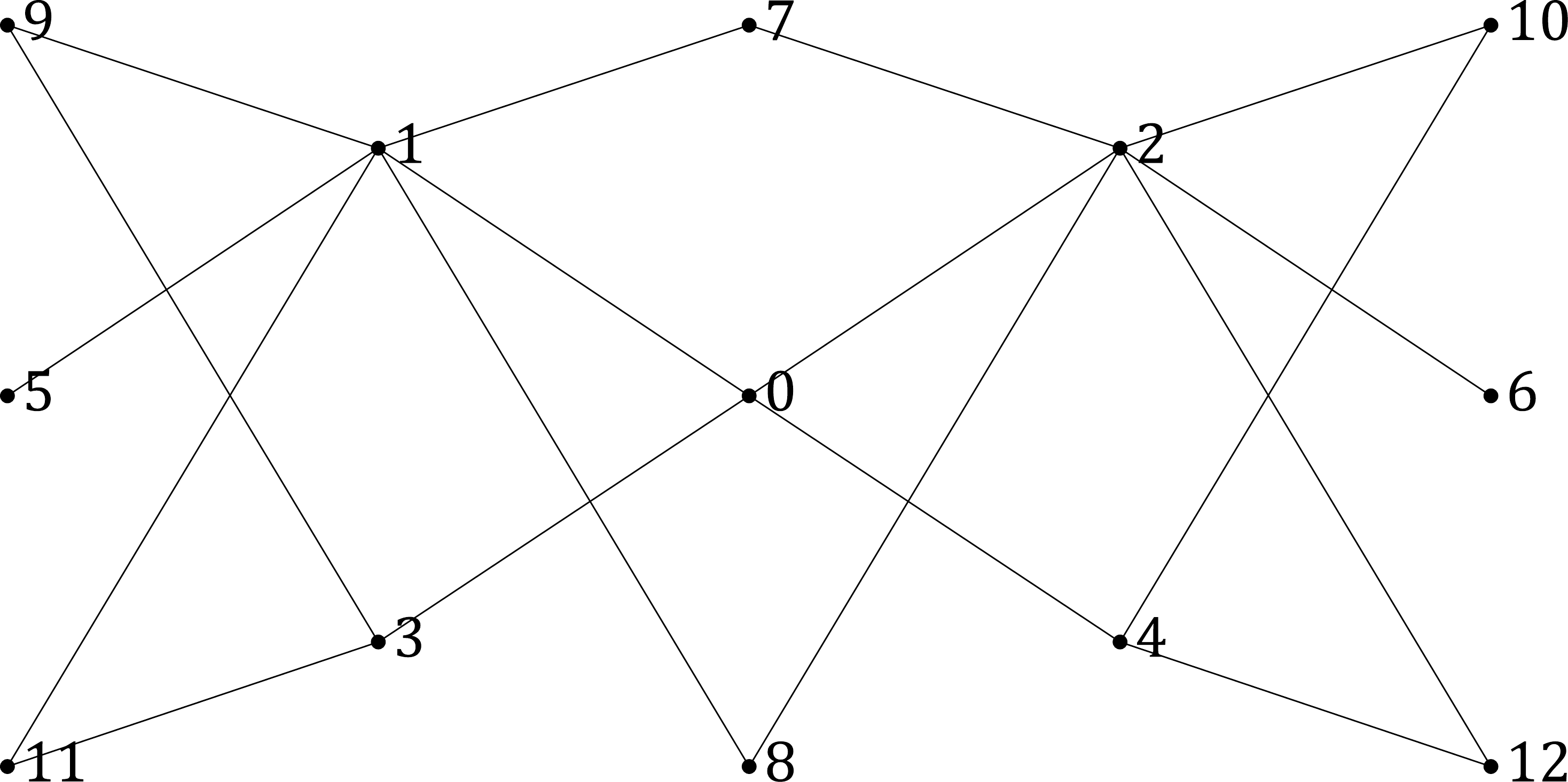}%
\caption{Flip connected components of the $7 \times 3 \times 2$ box, arranged by twist in increasing order. The numbering is the same as in table \ref{table:CC732}, and two dots $A$ and $B$ are connected if there exists a trit that takes a tiling in connected component $A$ to a tiling in connected component $B$. As we proved earlier, a trit from left to right in the diagram is always a positive trit; and a trit from right to left is always a negative one. }%
\label{fig:box732_CCdiagram}%
\end{figure}
\end{example}

\begin{example}[A region with two unequal floors] \label{example:unequalFloors}

\begin{figure}[ht]%
\centering
\includegraphics[width=0.6\columnwidth]{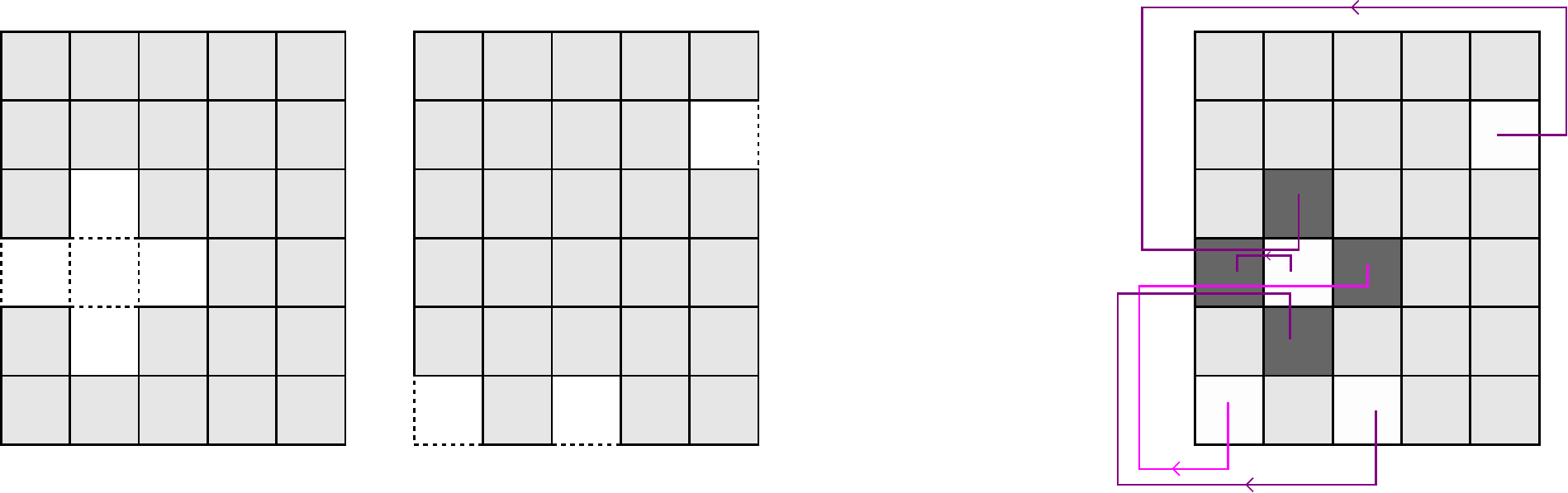}%
\caption{A region with two unequal floors, together with a choice of connections between the sources and the sinks. This choice of connections is the one used for the calculations in Table \ref{table:biggerRegion_CC}.}
\label{fig:bigger_region_sourcesAndSinks}%
\end{figure}

Figure \ref{fig:bigger_region_sourcesAndSinks} shows a region with two unequal floors, together with a choice of how to join sources and sinks. It has $642220$ tilings, and $30$ connected components. 

Table \ref{table:biggerRegion_CC} shows some information about these components and Figure \ref{fig:bigGraph_CC_unequal} shows a diagram of the flip connected components, arranged by their twists. This graph is not as symmetric as the one in the previous example; nevertheless, the space of tilings is also connected by flips and trits in this case.

\begin{table}[hpt]
\centering
\footnotesize
\begin{tabular}{|c|c|c|c|}
\hline
 \parbox[c]{2cm}{\centering Connected \\ Component} & \parbox[c]{2cm}{\centering Number of \\tilings} &$P_t(q)$ & $\Tw(t)$\\ \hline
 
0 & \num{165914} & $-2q -q^{-1}$ & $-1$ \\ \hline
1 & \num{153860} & $-q -1 -q^{-1}$ & $0$ \\ \hline
2 & \num{92123} & $-2q -1$ & $-2$ \\ \hline
3 & \num{56936} & $-q -1 -q^{-2}$ & $1$ \\ \hline
4 & \num{50681} & $-q -2$ & $-1$ \\ \hline
5 & \num{41236} & $-2q -q^{-2}$ & $0$ \\ \hline
6 & \num{17996} & $-2 -q^{-1}$ & $1$ \\ \hline
7 & \num{13448} & $-q -2q^{-1}$ & $1$ \\ \hline
8 & \num{11220} & $-3q$ & $-3$ \\ \hline
9 & \num{8786} & $-2 -q^{-2}$ & $2$ \\ \hline
10 & \num{7609} & $-q -q^{-1} -q^{-2}$ & $2$ \\ \hline
11 & \num{6423} & $-2q + 1 -2q^{-1}$ & $0$ \\ \hline
12 & \num{4560} & $-3q + 1 -q^{-1}$ & $-2$ \\ \hline
13 & \num{4070} & $-3$ & $0$ \\ \hline
14 & \num{3299} & $-2q + 1 -q^{-1} -q^{-2}$ & $1$ \\ \hline
15 & \num{2097} & $-1 -2q^{-1}$ & $2$ \\ \hline
16 & \num{1382} & $-1 -q^{-1} -q^{-2}$ & $3$ \\ \hline
17 & \num{221} & $-q + 1 -3q^{-1}$ & $2$ \\ \hline
18 & \num{137} & $-q + 1 -2q^{-1} -q^{-2}$ & $3$ \\ \hline
19 & \num{51} & $-3q^{-1}$ & $3$ \\ \hline
20 & \num{48} & $-2q -1$ & $-2$ \\ \hline
21 & \num{36} & $-2q^{-1} -q^{-2}$ & $4$ \\ \hline
22 & \num{17} & $-3q^{-1}$ & $3$ \\ \hline
23 & \num{17} & $-2q^{-1} -q^{-2}$ & $4$ \\ \hline
24 & \num{16} & $-1 -2q^{-1}$ & $2$ \\ \hline
25 & \num{16} & $-1 -q^{-1} -q^{-2}$ & $3$ \\ \hline
26 & \num{12} & $-2q + 2 -3q^{-1}$ & $1$ \\ \hline
27 & \num{7} & $-2q + 2 -2q^{-1} -q^{-2}$ & $2$ \\ \hline
28 & \num{1} & $1 -4q^{-1}$ & $4$ \\ \hline
29 & \num{1} & $1 -3q^{-1} -q^{-2}$ & $5$ \\ \hline

\end{tabular}
\caption{Information about the flip connected components of the region $R$ from Figure \ref{fig:bigger_region_sourcesAndSinks}.}
\label{table:biggerRegion_CC}
\end{table}

\begin{figure}[ht]%
\centering
\includegraphics[width=0.6\columnwidth]{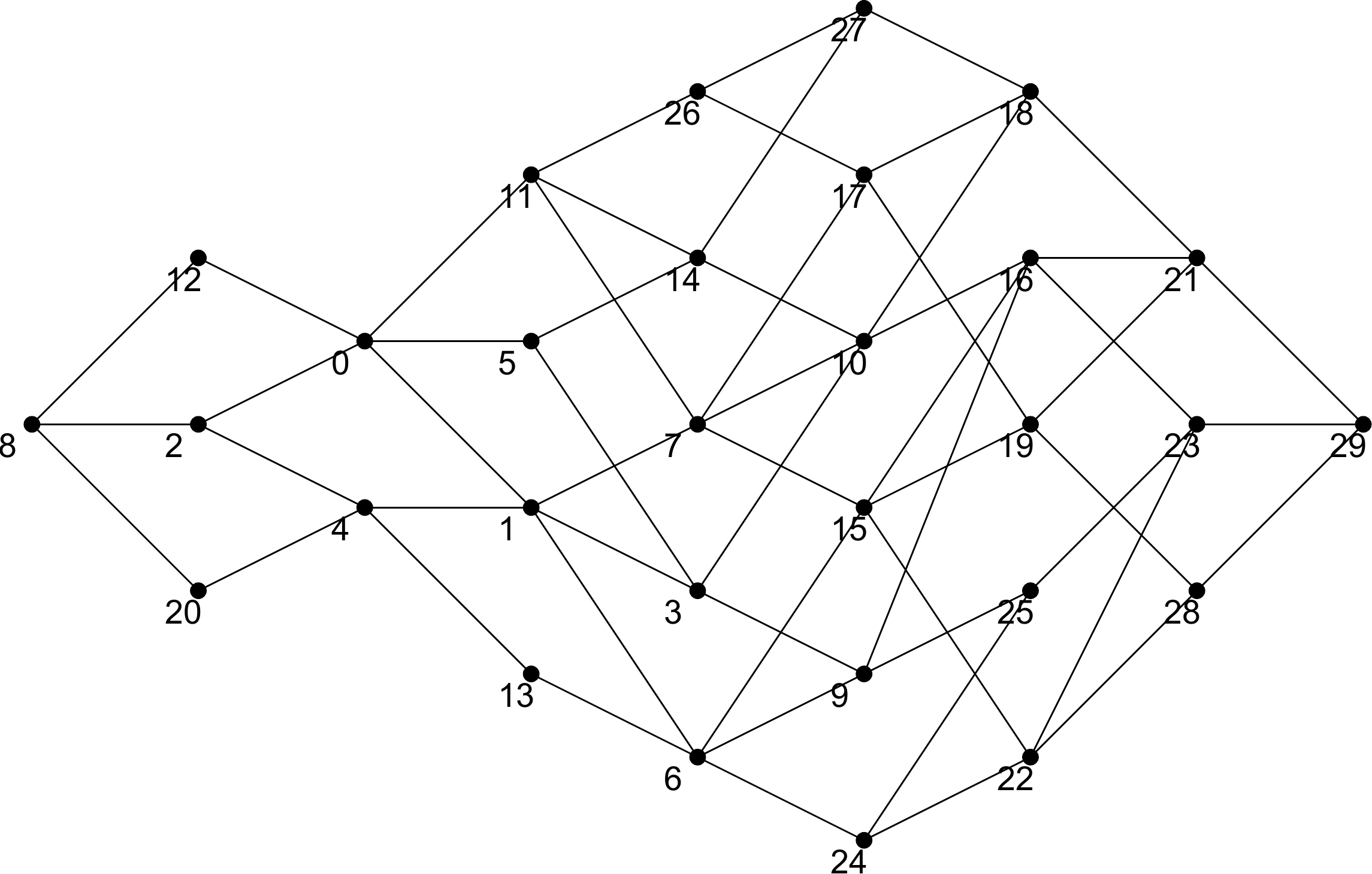}%
\caption{Graph with $30$ vertices, each one representing a connected component of the region. As in Figure \ref{fig:box732_CCdiagram}, two vertices are connected if there exists a trit taking a tiling in one component to a tiling in the other; a trit from left to right is always positive.}%
\label{fig:bigGraph_CC_unequal}%
\end{figure}

\end{example}

\section{The invariant in more space}
\label{sec:twoFloorsMoreSpace}
We already know that tilings that are not in the same flip connected component may have the same polynomial invariant, even in the case with two equal simply connected floors (see, for instance, Example \ref{example:732box}). However, as we will see in this section, this is rather a symptom of lack of space than anything else. 

More precisely, suppose $R$ is a duplex region and $t$ is a tiling of $R$. If $B$ is an $L \times M \times 2$ box (or a two-floored box) containing $R$, then $B \setminus R$ can be tiled in an obvious way (using only $z$ dimers). Thus $t$ induces a tiling $\hat{t}$ of $B$ that contains $t$; we call this tiling $\hat{t}$ the \emph{embedding} of $t$ in $B$.

\begin{figure}[ht]%
\centering
\includegraphics[width=0.5\columnwidth]{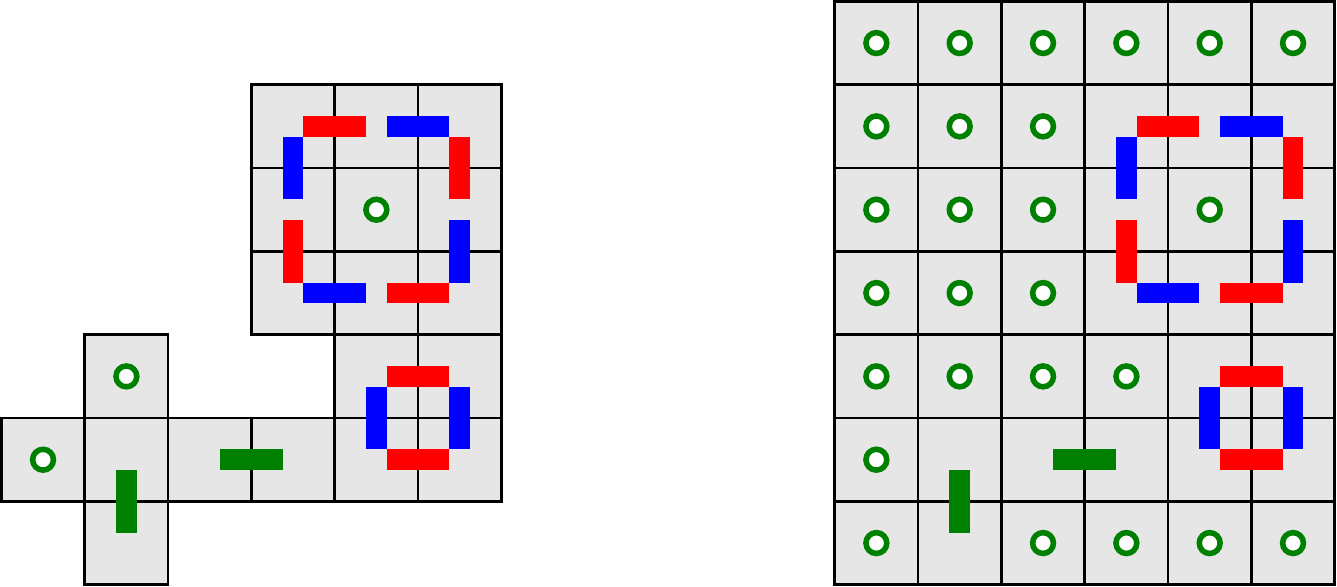}%
\caption{The associated drawing of a tiling, and its embedding in a $6 \times 7 \times 2$ box. }%
\label{fig:embeddingTwoFloors}%
\end{figure}

Another way to look at the embedding $\hat{t}$ of a tiling $t$ is the following: start with the associated drawing of $t$ (which is a plane region), add empty squares until you get an $L \times M$ rectangle, and place a jewel in every empty square. Since the newly added jewels are outside of any cycle in $\hat{t}$, it follows that $P_{\hat{t}}(q) - P_t(q) = k \in \ZZ$, where $k$ is the number of new black jewels minus the number of new white jewels (which depends only on the choice of box $B$ and not on the tiling $t$).

\begin{prop}
\label{prop:moreSpace}
Let $R$ be a duplex region, and let $t_0, t_1$ be two tilings that have the same invariant, i.e., $P_{t_0} = P_{t_1}$. Then there exists a two-floored box containing $R$ such that the embeddings $\hat{t_0}$ and $\hat{t_1}$ of $t_0$ and $t_1$ lie in the same flip connected component.  
\end{prop}

If $t_0$ and $t_1$ already lie in the same flip connected component in $R$, then their embeddings $\hat{t_0}$ and $\hat{t_1}$ in any two-floored box will also lie in the same connected component, because you can reach $\hat{t_1}$ from $\hat{t_0}$ using only flips already available in $R$.

Also, notice that $P_{\hat{t_0}} = P_{\hat{t_1}}$ if and only if $P_{t_0} = P_{t_1}$, because  $P_{\hat{t_1}}(q) - P_{t_1}(q) = P_{\hat{t_0}}(q) - P_{t_0}(q)$. Therefore, Proposition \ref{prop:moreSpace} states that two tilings have the same invariant if and only if there exists a two-floored box where their embeddings lie in the same flip connected component.

The reader might be wondering why we're restricting ourselves to duplex regions. One reason is that for general regions with two simply connected floors, it is not always clear that you can embed them in a large box in a way that their complement is tileable, let alone tileable in a natural way. 

Although it is technically possible to prove Proposition \ref{prop:moreSpace} only by looking at associated drawings, it will be useful to introduce an alternative formulation for the problem. 

Let $G = G(R)$ be the undirected plane graph whose vertices are the centers of the squares in the associated drawing of $R$, and where two vertices are joined by an edge if their Euclidean distance is exactly $1$. A \emph{system of cycles}, or \emph{sock}, in $G$ is a (finite) directed subgraph of $G$ consisting only of disjoint oriented (simple) cycles. An \emph{edge} of a sock is an (oriented) edge of one of the cycles, whereas a \emph{jewel} is a vertex of $G$ that is not contained in the system of cycles.

\begin{figure}%
\centering
\includegraphics[width=0.6\columnwidth]{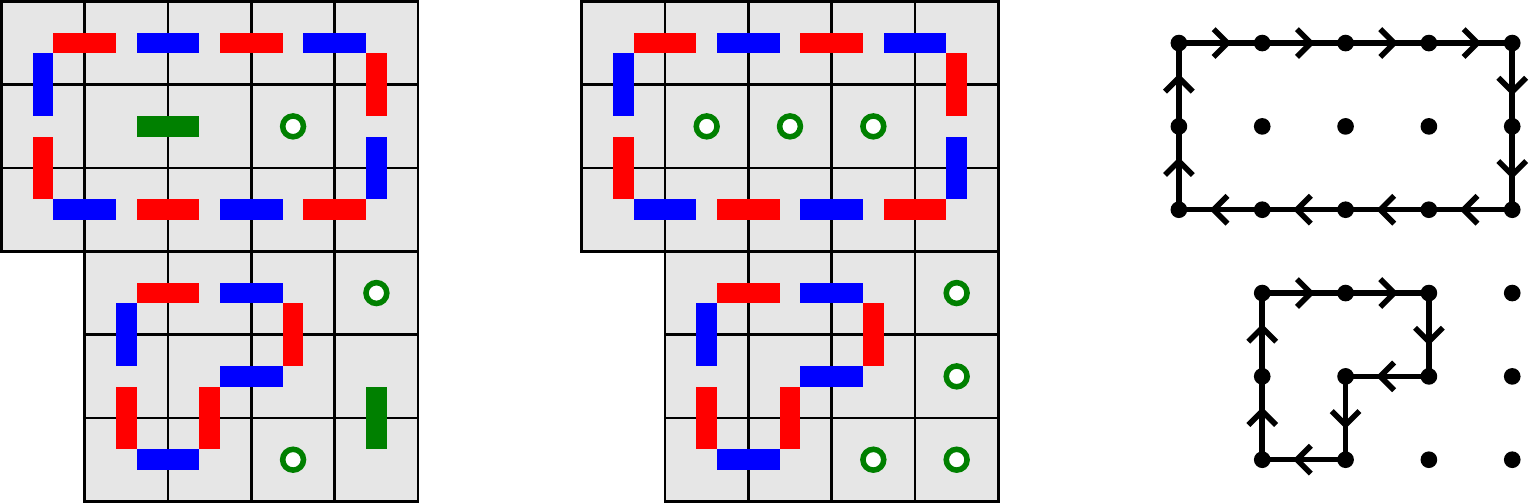}%
\caption{A tiling with two trivial cycles; the same tiling with the two trivial cycles flipped into jewels; and the system of cycles that corresponds to both of them.}%
\label{fig:tilingTwoFloors_withSOC}%
\end{figure}
There is an ``almost'' one-to-one correspondence between the systems of cycles of $G$ and the tilings of $R$, which is illustrated in Figure \ref{fig:tilingTwoFloors_withSOC}. In fact, tilings with trivial cycles have no direct interpretation as a system of cycles; but since all trivial cycles can be flipped into a pair of adjacent jewels, we can think that every sock represents a set of tilings, all in the same flip connected component.

We would now like to capture the notion of a flip from the world of tilings to the world of socks. This turns out to be rather simple: a \emph{flip move} on a sock is one of three types of moves that take one sock into another, shown in Figure \ref{fig:flipsteps}. Notice that performing a flip move on a sock corresponds to performing one or more flips on its corresponding tiling. 

\begin{figure}[ht]%
\centering
\def\svgwidth{0.3\columnwidth}
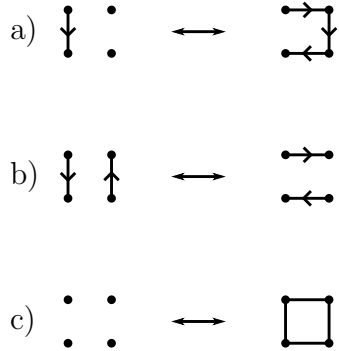
\caption{The three types of flip moves. The square in type c) is not oriented because it can have either one of the two possible orientations.}%
\label{fig:flipsteps}%
\end{figure}

A \emph{flip homotopy} in $G$ between two socks $s_1$ and $s_2$ is a finite sequence of flip moves taking $s_1$ into $s_2$. If there exists a flip homotopy between two socks, they are said to be \emph{flip homotopic} in $G$. Notice that two tilings are in the same flip connected component of $R$ if and only if their corresponding socks are flip homotopic in $G$, because every flip can be represented as one of the flip moves (and the flip that takes a trivial cycle into two jewels does not alter the corresponding sock). Figure \ref{fig:flipStepsExample} shows examples of flips and their corresponding flip moves.

\begin{figure}[ht]%
\centering
\includegraphics[width=0.7\columnwidth]{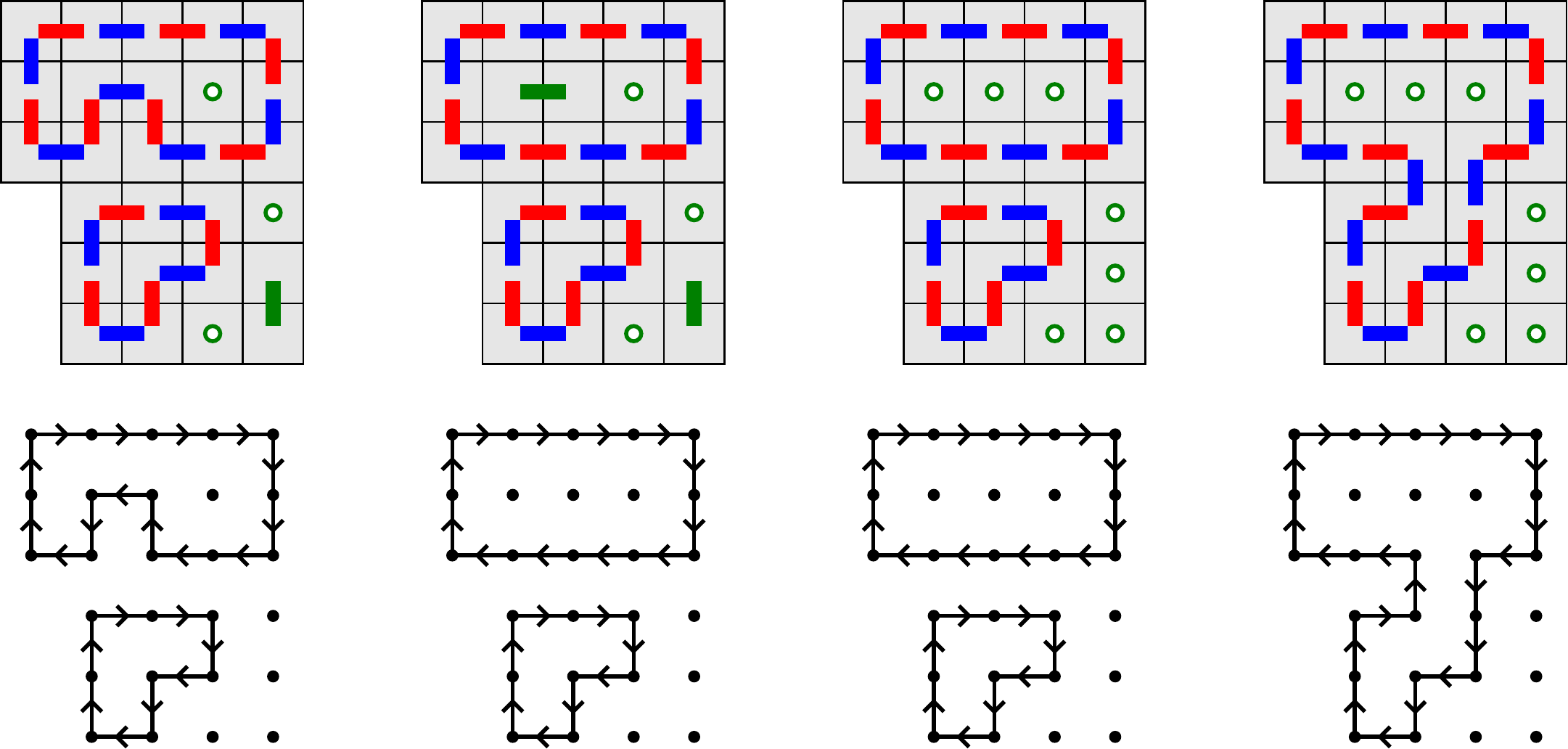}%
\caption{Examples of how flips affect the corresponding sock. The first flip induces a flip move of type (a), The second one does not alter the corresponding sock, and the third one induces a flip move of type (b). }%
\label{fig:flipStepsExample}%
\end{figure}

One advantage of this new interpretation is that we can easily add as much space as we need without an explicit reference to a box. In fact, notice that $G$ is a subgraph of the infinite graph $\ZZ^2$, so that a sock in $G$ is also a finite subgraph of $\ZZ^2$, so that we may see it as a system of cycles in $\ZZ^2$.

\begin{lemma}
\label{lemma:equivEmbeddingHomotopy}
For two tilings $t_0$ and $t_1$ of a region $R$, the following assertions are equivalent:
\begin{enumerate}[label=(\roman*)]
	\item \label{item:embeddedBox} There exists a two-floored box $B$ containing $R$ such that the embeddings of $t_0$ and $t_1$ in $B$ lie in the same connected component.
	\item \label{item:flipHomotopicSOC} The corresponding systems of cycles of $t_0$ and $t_1$ are flip homotopic in $\ZZ^2$.
\end{enumerate}  
\end{lemma}
\begin{proof}
To see that \ref{item:flipHomotopicSOC} implies \ref{item:embeddedBox}, notice the following:
if $s_0, s_1, \ldots, s_n$ are the socks involved in the flip homotopy between the socks of $t_0$ and $t_1$ in $\ZZ^2$, let $B$ be a sufficiently large two-floored box such that $(\ZZ^2 \setminus G(B))$ contains only vertices of $\ZZ^2$ that are jewels in all the $n+1$ socks $s_0, s_1, \ldots, s_n$. Then the socks of $t_0$ and $t_1$ are flip homotopic in $G(B)$, so the embeddings of $t_0$ and $t_1$ lie in the same flip connected component.

The converse is obvious, since if the socks of $t_0$ and $t_1$ are flip homotopic in $G(B)$ for some two-floored box $B$, then they are flip homotopic in $\ZZ^2$.
\end{proof} 

A vertex $v \in \ZZ^2$ is said to be white (resp. black) if the sum of its coordinates is even (resp. odd). If $s$ is a system of cycles in $\ZZ^2$, we define the graph invariant of $s$ as $$P_s(q) = \sum_{\substack{j \scalebox{0.7}{\mbox{ black jewel}}\\ k_s(j) \neq 0}}q^{k_s(j)} - \sum_{\substack{j \scalebox{0.7}{\mbox{ white jewel}}\\ k_s(j) \neq 0}}q^{k_s(j)},$$
where $k_s(j)$ is the sum of the winding numbers of all the cycles in $s$ (as curves) with respect to $j$; this is a finite sum because there is a finite number of jewels $j$ for which $k_s(j) \neq 0$: they are among those jewels that are enclosed by cycles of $s$. Notice that if $t$ is a tiling of $R$ and $s$ is its corresponding sock in $\ZZ^2$, $P_t(q) - P_s(q) = P_t(1) - P_s(1)$, which equals the $q^0$ term in $P_t$. Since $P_t(1)$ equals the number of black squares minus the number of white squares in $R$ (thus does not depend on $t$), it follows that $P_t(q)$ is completely determined by $P_s(q)$.

A corollary of Proposition \ref{prop:twoFloorFlipInvariant} is that if two systems of cycles $s_0$ and $s_1$ are flip homotopic in $\ZZ^2$, then $P_{s_0} = P_{s_1}$. We now set out to prove that the converse also holds, which will establish Proposition \ref{prop:moreSpace}.

A \emph{boxed jewel} is a subgraph of $\ZZ^2$ formed by a single jewel enclosed by a number of square cycles (cycles that are squares when thought of as plane curves), all with the same orientation. Figure \ref{fig:boxedJewelsExample} shows examples of boxed jewels. Working with boxed jewels is easier, for if they have ``free space'' in one direction (for instance, if there are no cycles to the right of it), they can move an arbitrary even distance in that direction; the simplest case is illustrated in Figure \ref{fig:boxedJewelsMove}. More complicated boxed jewels move just as easily: we first turn the outer squares into rectangles, then we move the inner boxed jewel, and finally we close the outer squares again.

\begin{figure}[ht]%
\centering
\subfloat[]{\includegraphics[width=0.06\columnwidth]{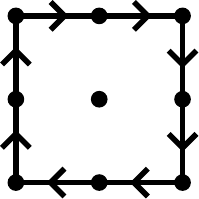}}\qquad
\subfloat[]{\includegraphics[width=0.24\columnwidth]{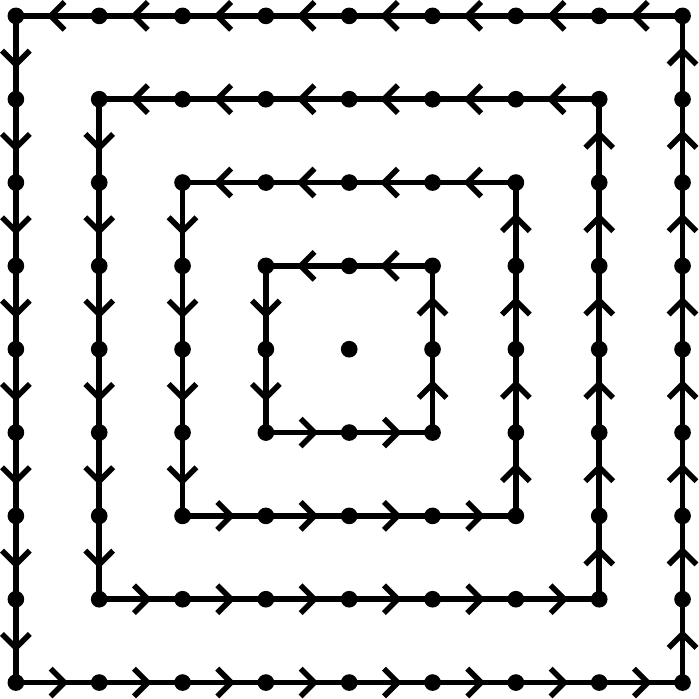}}
\caption{Two examples of boxed jewels.}%
\label{fig:boxedJewelsExample}%
\end{figure} 

\begin{figure}%
\centering
\includegraphics[width=0.7\columnwidth]{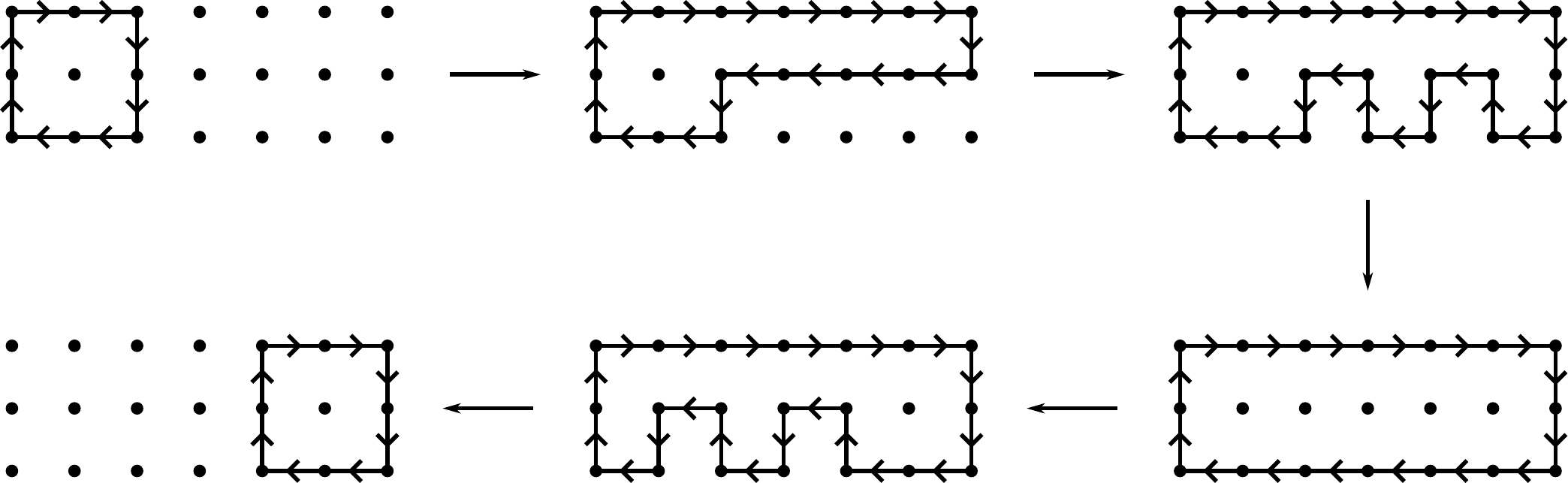}%
\caption{A boxed jewel with ``free space'' to the right. Starting from the first sock, we perform: four flip moves of type (a); two flip moves of type (a); two flip moves of type (a) to create a rectangle; two flip moves of type (a); and finally, the last sock is obtained by performing six flip moves of type (a) on the penultimate sock.}%
\label{fig:boxedJewelsMove}%
\end{figure}

An \emph{untangled} sock is a sock that contains only boxed jewels, and such that the center of each boxed jewel is of the form $(n, 0)$ for some $n \in \ZZ$ (that is, all the enclosed jewels lie in the $x$ axis), as illustrated in Figure \ref{fig:untangledSockExample}. Therefore, each boxed jewel in an untangled sock moves very easily: it has free space both downwards and upwards.

\begin{figure}%
\centering
\includegraphics[width=0.6\columnwidth]{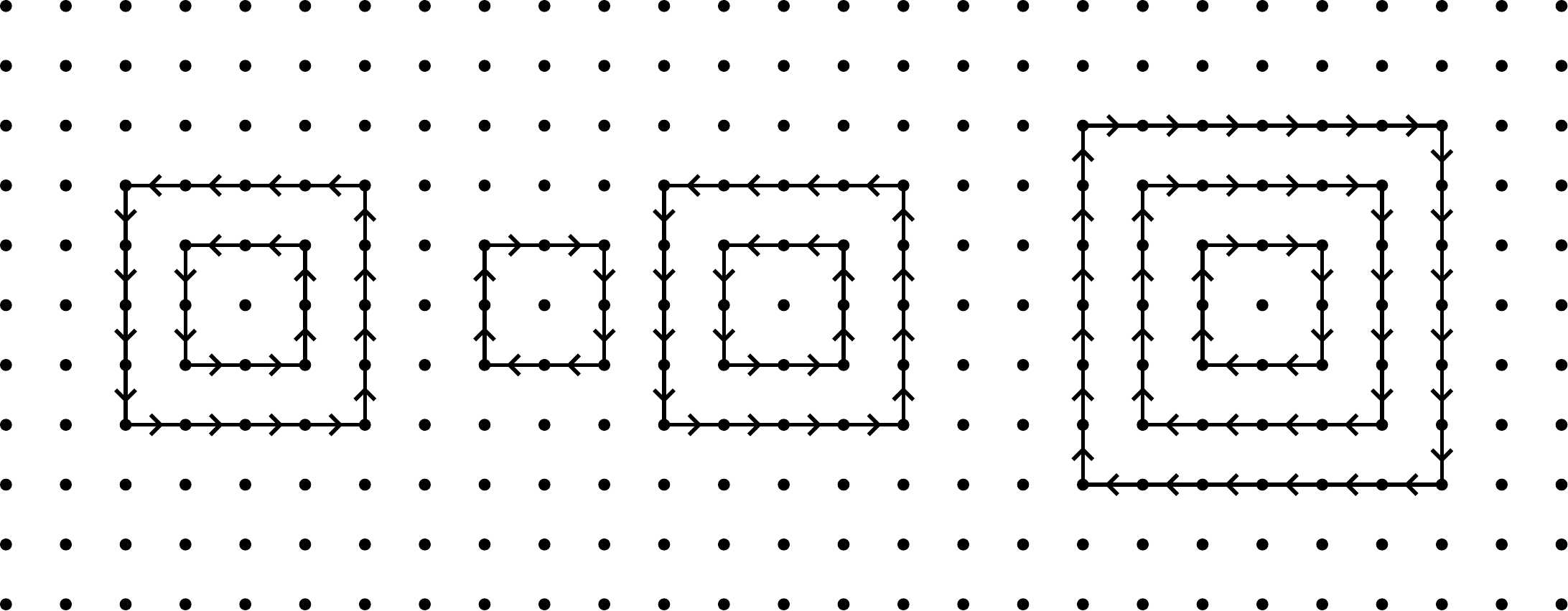}%
\caption{Example of an untangled sock.}%
\label{fig:untangledSockExample}%
\end{figure}


\begin{lemma}
\label{lemma:untangledSocksInvariant}
Two untangled socks that have the same invariant are flip homotopic in $\ZZ^2$.
\end{lemma}
\begin{proof}
If the two socks consist of precisely the same boxed jewels but in a different order, then they are clearly flip homotopic, since we can easily move the jewels around and switch their positions as needed. We only need to check that boxed jewels that cancel out (that is, they refer to terms with the same exponent but opposite signs) can be ``dissolved'' by flip moves.

\begin{figure}%
\centering
\includegraphics[width=0.6\columnwidth]{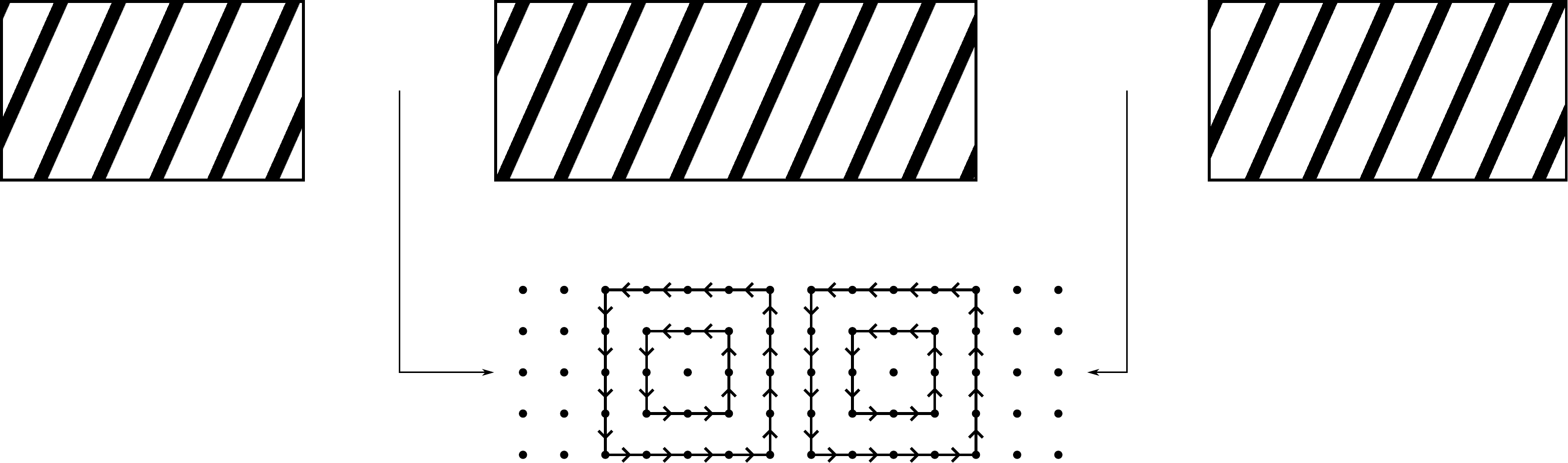}%
\caption{Illustration of how boxed jewels that cancel out can be brought close. The striped rectangles indicate areas where there may be other boxed jewels.}%
\label{fig:boxedJewelsCancellation}%
\end{figure}
\begin{figure}[ht]%
\centering
\includegraphics[width=0.6\columnwidth]{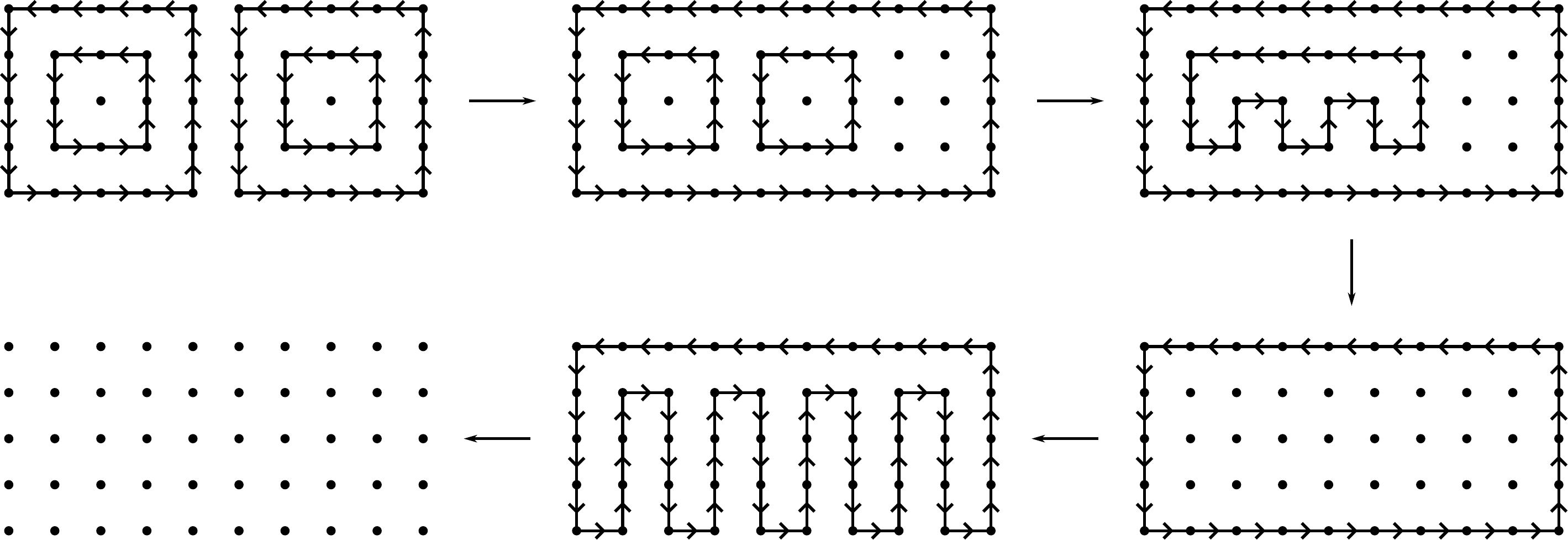}%
\caption{Some intermediate steps in the flip homotopy that ``dissolves'' a pair of cancelling boxed jewels}%
\label{fig:boxedJewelsCancellationSteps}%
\end{figure}

In order to see this, start by moving the boxed jewels that cancel out downwards, then toward each other until they are next to each other, as illustrated in Figure \ref{fig:boxedJewelsCancellation}. Once they are next to each other, they can be elliminated by a sequence of flip moves. Figure \ref{fig:boxedJewelsCancellationSteps} shows some of the steps involved in the flip homotopy that eliminates this pair of cancelling jewels. 
 
\end{proof}

If $s$ is a sock, we define the \emph{area} of $s$ to be the sum of the areas enclosed by each cycle of $s$, thought of as a plane curve. The areas count as positive regardless of the orientation of the cycles. As an example, the boxed jewels shown in Figure \ref{fig:boxedJewelsExample} have areas $4$ and $120,$ respectively, and the untangled sock in Figure \ref{fig:untangledSockExample} has area $20 + 4 + 20 + 56 = 100$. Naturally, the only sock with zero area is the empty sock (the sock with no cycles).

The following Lemma is the key step in the proof of Proposition \ref{prop:moreSpace}:

\begin{lemma}
\label{lemma:homotopicToUntangled}
Every sock is flip homotopic to an untangled sock in $\ZZ^2$.
\end{lemma}
\begin{proof}
Suppose, by contradiction, that there exists a sock which is not flip homotopic to an untangled sock. Of all the examples of such socks, pick one, $s_0$, that has minimal area (which is greater than zero, because the empty sock is already untangled).

Among all the cycles in $s_0$, consider the ones who have vertices that are furthest bottom. Among all these vertices, pick the rightmost one, which we will call $v$. In other words, assuming the axis are as in Figure \ref{fig:notation2Dexample}: if $m = \max \{n : (k,n) \mbox{ is a vertex of } s_0 \mbox{ for some } k \in \ZZ\}$ and $l = \max \{k : (k,m) \mbox{ is a vertex of } s_0\}$, then $v = (l,m)$. 

Clearly $v$ is the right end of a horizontal edge, and the bottom end of a vertical edge, as portrayed in Figure \ref{fig:rightmostSouthmostPoint}. We may assume without loss of generality that these edges are oriented as in the aforementioned Figure.    
\begin{figure}[ht]%
\centering
\includegraphics[width=0.25\columnwidth]{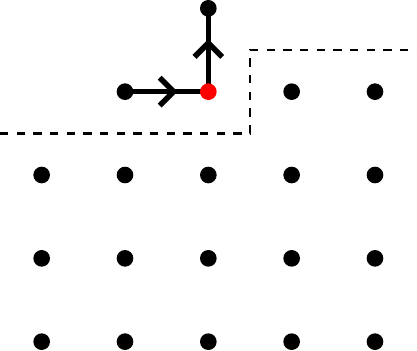}%
\caption{The red vertex indicates the rightmost bottommost vertex in the nonempty sock $s_0$. By definition, there can be no cycle parts in $s_0$ below the dotted line.}%
\label{fig:rightmostSouthmostPoint}%
\end{figure}

Consider the diagonal of the form $v - (n,n), n \geq 0, n \in \ZZ$, starting from $v$ and pointing northwest, and let $w$ be the first point (that is, the one with the smallest $n$) in this diagonal that is not the right end of a horizontal edge pointing to the right and the bottom end of a vertical edge pointing up (see Figure \ref{fig:firstInTheDiagonal}).

\begin{figure}[ht]%
\centering
\includegraphics[width=0.4\columnwidth]{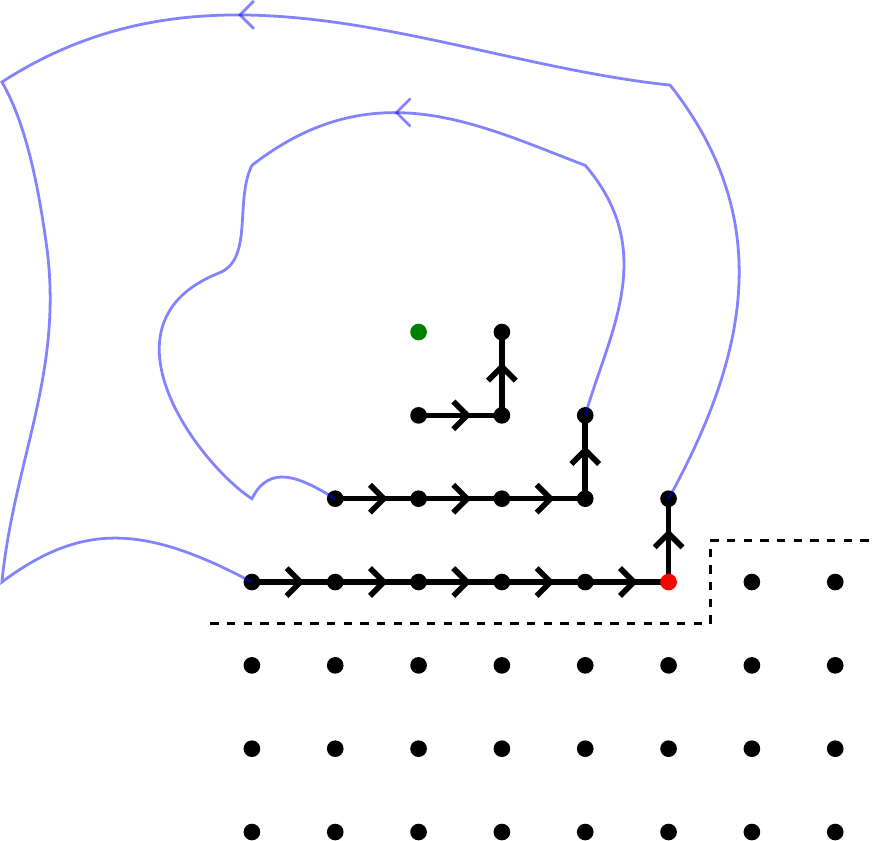}%
\caption{The red vertex represents $v$, while the green one represents $w$, which is the first vertex in the diagonal that does not follow the pattern (``edgewise'') of the other three. The blue segments represent (schematically) the relative positions of two of the cycles, which must be in this way because there can be no cycle parts below the dashed line. }%
\label{fig:firstInTheDiagonal}%
\end{figure}

We will now see that if $w$ is not a jewel, then we can immediately reduce the area of $s_0$ with a single flip move. This leads to a contradiction, because the sock obtained after this flip move cannot be flip homotopic to an untangled sock, but has smaller area than $s_0$.

Suppose $w$ is not a jewel. With Figure \ref{fig:firstInTheDiagonal} in mind, consider the vertex directly below $w$. It is either the bottom end of a vertical edge pointing downward (case 1), or the right end of a horizontal edge pointing to the right. In case 1, a single flip move of type (a) will immediately reduce the area, as shown in the first drawing in Figure \ref{fig:firstInTheDiagonal_notJewel}. The other three drawings handle all the possibilities for the other case: notice that in each case there is a flip move that reduces the area.

\begin{figure}[ht]%
\centering
\includegraphics[width=0.7\columnwidth]{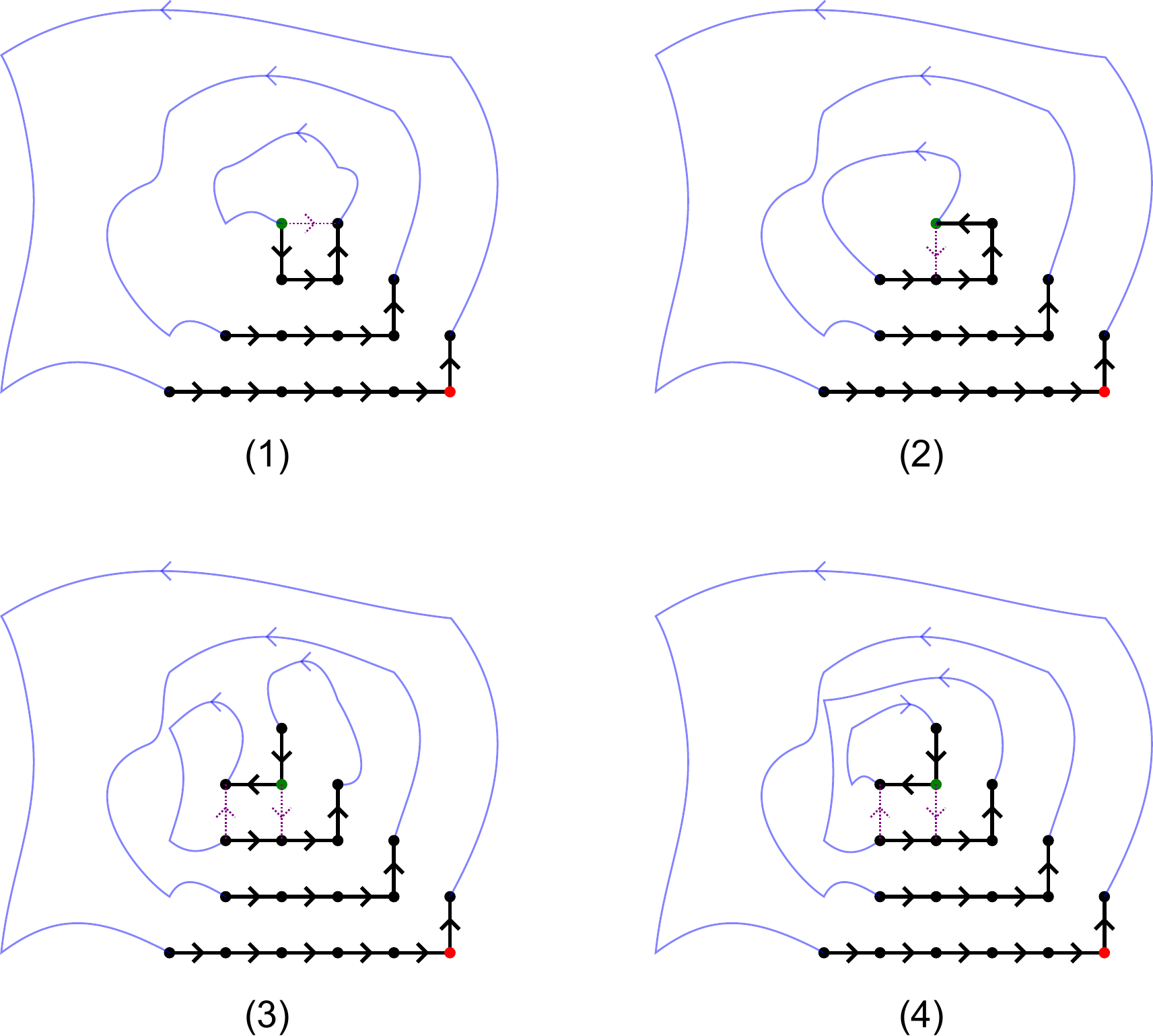}%
\caption{The four possibilities when $w$ is not a jewel. Notice that, in each case, there is a flip move that reduces the area, which are indicated by the dotted purple lines.}%
\label{fig:firstInTheDiagonal_notJewel}%
\end{figure} 

The most interesting case is when $w$ is a jewel. The key observation here is that the jewel may be then ``extracted'' from all the cycles as a boxed jewel, and what remains has smaller area. Figure \ref{fig:jewelExtraction} illustrates the steps involved in extracting a jewel.

\begin{figure}[ht]%
\centering
\includegraphics[width=\columnwidth]{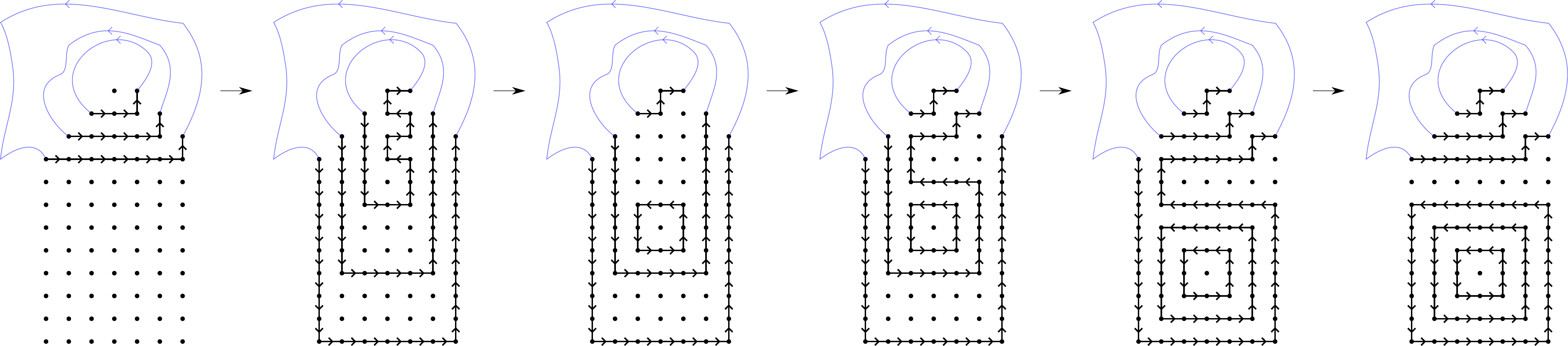}%
\caption{Some intermediate steps in the extraction of a jewel. Notice that this flip homotopy reduced the area of the cycles that previously enclosed the jewel.}%
\label{fig:jewelExtraction}%
\end{figure}

Suppose $w$ is a jewel, and consider the flip homotopy that extracts $w$ (as in Figure \ref{fig:jewelExtraction}). Let $s_1$ be the sock obtained at the end of this flip homotopy, and let $\tilde{s_1} = \tilde{s_2}$ be the sock consisting of all the cycles of $s_1$ except the extracted boxed jewel. Clearly the area of $\tilde{s_2}$ is less than the area of $s_0$, and since $s_0$ is a sock that has minimal area among those that are not flip homotopic to an untangled sock, it follows that there exists a flip homotopy, say  $\tilde{s_2},\tilde{s_3},\ldots, \tilde{s_n}$ such that $\tilde{s_n}$ is an untangled sock.

Let $M = \max \{k : (l,k) \mbox{ is a vertex of } \tilde{s_i} \mbox{ for some } k \in \ZZ, 2 \leq i \leq n\}$. Recall that boxed jewels can move an arbitrary even distance if they have free space in some direction; by the definition of $v$, there can be no cycles or cycle parts below $v$ in $s_0$, so we can define $s_2$ (homotopic to $s_1$) by pulling the extracted jewel in $s_1$ down as much as we need, namely so that all vertices of the boxed jewel have $y$ coordinate strictly larger than $M$. Notice that $\tilde{s_2}$ is obtained from $s_2$ by removing this boxed jewel. Clearly we can perform on $s_2$ all the flip moves that took $\tilde{s_2}$ to $\tilde{s_n}$, so that we obtain a flip homotopy $s_2, s_3, \ldots, s_n$, where $s_n$ consists of all the cycles in $\tilde{s_n}$ plus a single boxed jewel down below. The other cycles in $s_n$ are also boxed jewels whose centers have $y$ coordinate equal to zero, because $\tilde{s_n}$ is untangled; therefore, the boxed jewel down below can be brought up and sideways as needed, thus obtaining a flip homotopy from $s_0$ to an untangled sock $s_{n+1}$. This contradicts the initial hypothesis, and thus the proof is complete.     
\end{proof}

This proof also yields an algorithm for finding the flip homotopy from an arbitrary sock to an untangled sock, although probably not a very efficient one. Start with the initial sock, and find $v$ and $w$ as in the proof. If $w$ is not a jewel, perform the flip move that reduces the area, and recursively untangle this new sock. If $w$ is a jewel, extract the boxed jewel, recursively untangle the sock without the boxed jewel (which has smaller area), and calculate how much space you needed to solve it: this will tell how far down the boxed jewel needs to be pulled. The algorithm stops recursing when we reach a sock with zero area: all that's left to do then is to ``organize'' the boxed jewels.

Together, Lemmas \ref{lemma:equivEmbeddingHomotopy}, \ref{lemma:untangledSocksInvariant} and \ref{lemma:homotopicToUntangled} establish Proposition \ref{prop:moreSpace}, which was the goal for this section. As a conclusion, the polynomial invariant here presented is, in a sense, complete: if two tilings have the same invariant and sufficient space is added to the region, then there is a sequence of flips taking one to the other. 
  
\section{Connectivity by flips and trits}
\label{sec:connectedFlipsTrits}
We introduced in this paper a new move, which we called a trit, that we defined in Section \ref{sec:defsAndNotations} (see Figure \ref{fig:postrit}). Also, in Section \ref{sec:effectOfTritsOnPt} we studied the effect of a trit in our invariant $P_t$.

We also pointed out that the graph of connected components for Examples \ref{example:732box} and \ref{example:unequalFloors}, shown in Figures \ref{fig:box732_CCdiagram} and \ref{fig:bigGraph_CC_unequal}, is connected in both cases. There, two components are joined if there exists a trit taking a tiling in one component to a tiling in the other. Hence, this graph is connected for a region if and only if for any two tilings of this region, one can reach one from the other via flips and trits.

A natural question is, therefore: for regions with two simply connected floors, is it true that one can always reach a tiling from any other via flips and trits? The answer is, in general, no, as Figure \ref{fig:unequal_twofloors_notConnectedFlipsTrits} shows. 

\begin{figure}%
\centering
\subfloat[]{\includegraphics[width=0.21\columnwidth]{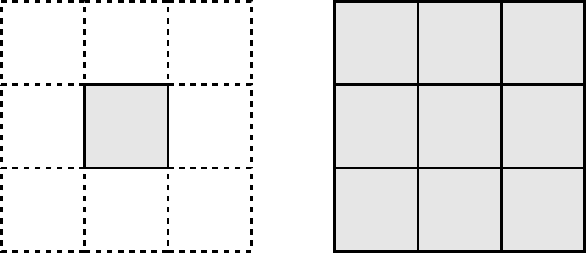}}\qquad
\subfloat[]{\includegraphics[width=0.33\columnwidth]{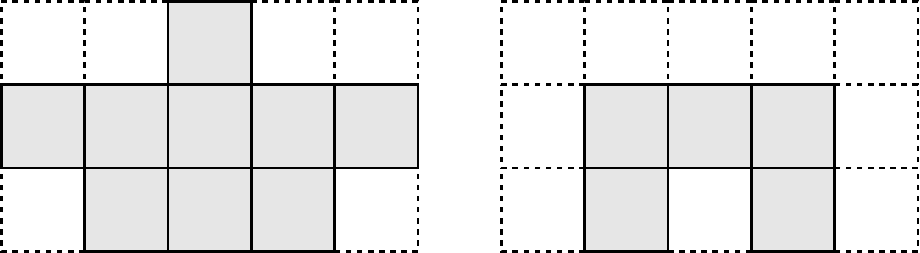}}
\caption{Two regions, each with two tilings where neither a flip nor a trit is possible.}%
\label{fig:unequal_twofloors_notConnectedFlipsTrits}%
\end{figure} 

Nevertheless, this is true for duplex regions:
\begin{prop}
\label{prop:connectivityFlipsTrits}
If $R$ is a duplex region and $t_0,t_1$ are two tilings of $R$, there exists a sequence of flips and trits taking $t_0$ to $t_1$.
\end{prop} 

In order to prove this result, we'll once again make use of the concept of systems of cycles, or socks, that were introduced in Section \ref{sec:twoFloorsMoreSpace}.

Understanding the effect of a trit on a sock turns out to be quite easy: in fact, the effect of a trit can be captured to the world of socks via the insertion of a new move, which we will call the \emph{trit move} (in addition to the three flip moves shown in Figure \ref{fig:flipsteps}). The trit move is shown in Figure \ref{fig:tritstep}.

\begin{figure}[ht]%
\centering
\includegraphics[width=0.3\columnwidth]{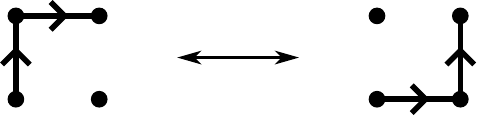}%
\caption{The trit move.}%
\label{fig:tritstep}%
\end{figure}

Another way to look at trit moves is that it either pulls a jewel out of a cycle or pushes one into a cycle. Two socks $s_1$ and $s_2$ are \emph{flip and trit homotopic} in a graph $G$ (which contains both $s_1$ and $s_2$) if there is a finite sequence of flip and/or trit moves taking $s_1$ to $s_2$. Notice that, unlike in Section \ref{sec:twoFloorsMoreSpace}, where we were mainly interested in flip homotopies in $\ZZ^2$, we are now interested in flip and trit homotopies in the finite graphs $G(R)$. Recall from Section \ref{sec:twoFloorsMoreSpace} that if $R$ is a duplex region, $G(R)$ is the planar graph whose vertices are the centers of the squares in the associated drawing of $R$, and where two vertices are joined by an edge if their Euclidean distance is exactly $1$.

Again, for two tilings $t_0, t_1$, there exists a sequence of flips and trits taking $t_0$ to $t_1$ if and only if their corresponding socks are flip and trit homotopic in $G(R)$.

\begin{lemma} \label{lemma:flipAndTritHomotopy} If $R$ is duplex region, and $s$ is a sock in $G(R)$, then $s$ is flip and trit homotopic to the empty sock (the sock with no cycles) in $G(R)$.
\end{lemma}
\begin{proof}
Suppose, by contradiction, that there exists a sock contained in $G(R)$ that is not flip and trit homotopic to the empty sock in $G(R)$. Of all the examples of such socks, let $s$ be one with minimal area (the concept of area of a sock was defined in Section \ref{sec:twoFloorsMoreSpace}). We will show that there exists either a flip move or a trit move that reduces the area of $s$, which is a contradiction.

Let $\gamma$ be a cycle of $s$ such that there is no other cycle inside $\gamma$: hence, if there is any vertex inside $\gamma$, it must be a jewel. Similar to the proof of Lemma \ref{lemma:homotopicToUntangled}, let $v$ be the rightmost among the bottommost vertices of $\gamma$ (notice that we are only considering the vertices of $\gamma$, and not all the vertices in $s$), or, in other words: if $m = \max \{n : (k,n) \mbox{ is a vertex of } \gamma \mbox{ for some } k \in \ZZ\}$ and $l = \max \{k : (k,m) \mbox{ is a vertex of } \gamma\}$, let $v = (l,m)$. Notice that $v$ is the right end of a horizontal edge and the bottom end of a vertical edge: we may assume without loss of generality that this horizontal edge points to the right. 

\begin{figure}[ht]%
\centering
\includegraphics[width=0.45\columnwidth]{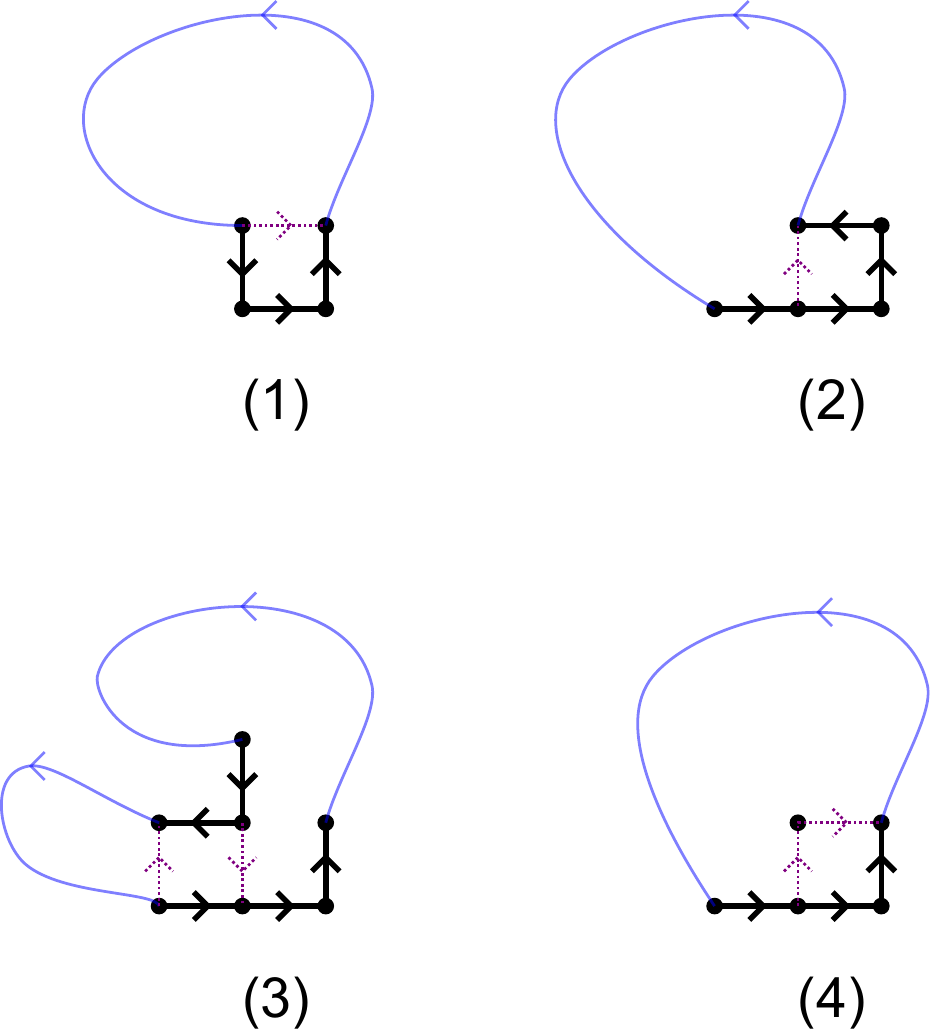}%
\caption{The four cases in the proof of Lemma \ref{lemma:flipAndTritHomotopy}, where the blue line indicates schematically the position of the edges in $\gamma$ that are not directly portrayed. The move that reduces the area is indicated by the dashed purple line: in the first three cases, it is a flip move; the last one is a trit move.}%
\label{fig:tritAndFlipHomotopy_cases}%
\end{figure}

Notice that the vertex $w = v - (1,1)$ is necessarily in the graph $G(R)$, because otherwise $\gamma$ would have a hole inside, which contradicts the hypothesis that the (identical) floors of $R$ are simply connected. Suppose first that $w$ is not a jewel. Since there are no cycles inside $\gamma$, it follows that $w$ must be a vertex of $\gamma$. If $w$ is the topmost end of a vertical edge poiting downward, we have the first case in Figure \ref{fig:tritAndFlipHomotopy_cases}, where we clearly have a flip move that reduces the area. If this is not the case, it follows, in particular, that $u = v - (2,0)$ must be in the graph $G(R)$: Cases (2) and (3) of Figure \ref{fig:tritAndFlipHomotopy_cases} show the two possibilities for the edges that are incident to $w$, and it is clear that there exists a flip move which reduces the area of $s$.

Finally, the case where $w$ is a jewel is shown in case (4) of Figure \ref{fig:tritAndFlipHomotopy_cases}: the available trit move clearly reduces the area. Hence, there is always a flip or trit move that reduces the area of the sock, which contradicts the minimality of the area of $s$.    
\end{proof}

Therefore, we have established Proposition \ref{prop:connectivityFlipsTrits}. Another observation is that the above proof also yields an algorithm for finding the flip and trit homotopy from a sock to the empty sock: while there is still some cycle in the sock, find one cycle that contains no other cycle. Then find $v$, as in the proof, and do the corresponding flip or trit move, depending on the case. Since each move reduces the area, it follows that we'll eventually reach the only sock with zero area, which is the empty sock.

\section{The invariant when more floors are added}
\label{sec:embedFourFloors}
In this section, we'll discuss the fact that the invariant $P_t(q)$ is not preserved when the tilings are embedded in ``big'' regions with more than two floors.

As in Section \ref{sec:twoFloorsMoreSpace}, we'll consider duplex regions. Recall from that section that we defined the embedding of such a tiling in a two-floored box $B$. Here, we'll extend this notion to embeddings in boxes with four floors in a rather straightforward manner.

If $t$ is a tiling of a duplex region $R$, and $B$ is a box with four floors such that $R$ is contained in the top two floors of $B$, then the embedding $\hat{t}$ of $t$ in $B$ is the tiling obtained by first embedding $t$ in the top two floors of $B$ (as in Section \ref{sec:twoFloorsMoreSpace}), and then filling the bottom two floors with ``jewels'' (that is, dimers connecting both floors). This is illustrated in Figure \ref{fig:embeddingFourFloorsExample}.

\begin{figure}[ht]%
\centering
\includegraphics[width=0.5\columnwidth]{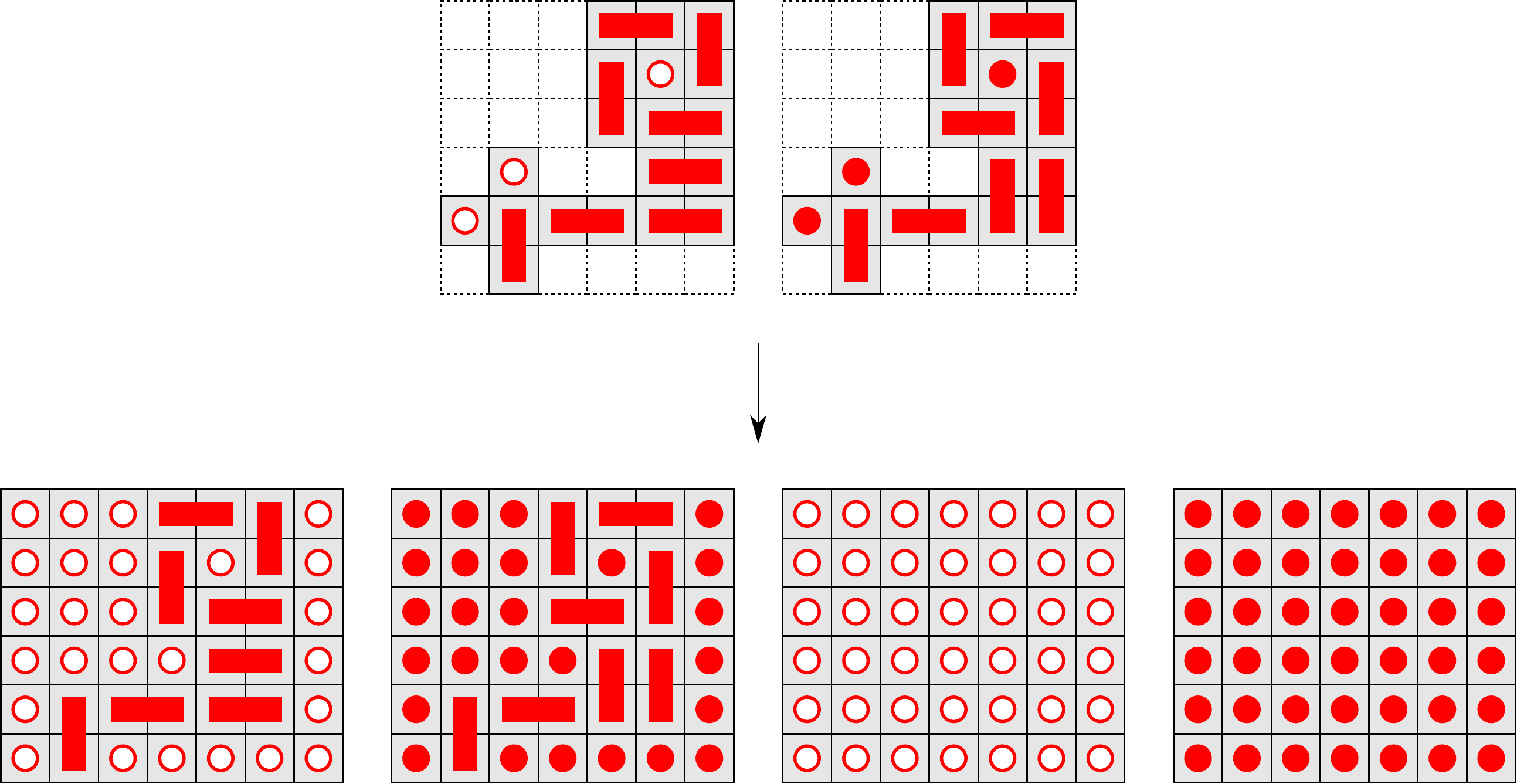}%
\caption{A tiling of a two-floored region and its embedding in a $7 \times 6 \times 4$ box.}%
\label{fig:embeddingFourFloorsExample}%
\end{figure}

\begin{prop}
\label{prop:embedFourFloors}
If $t_0$ and $t_1$ are two tilings of a duplex region and $P_{t_0}'(1) = P_{t_1}'(1)$, then there exists a box with four floors $B$ such that the embeddings of $t_0$ and $t_1$ in $B$ lie in the same flip connected component.
\end{prop}

The converse also holds, but the proof requires some technology from \cite{segundoartigo}. 

Proposition \ref{prop:embedFourFloors} is essentially stating that the invariant $P_t(q)$ no longer survives when more floors are added, although the twist $\Tw(t)$ still does; this fact was briefly mentioned before, and will be discussed in \cite{segundoartigo}.

Recall from Section \ref{sec:twoFloorsMoreSpace} the definitions of boxed jewel and sock. The key fact in the proof of Proposition \ref{prop:embedFourFloors} is that a boxed jewel associated to a $q^n$ term in $P_t(q)$ can be transformed into a number of smaller boxed jewels, their terms adding up to $nq$. In what follows, the \emph{sign} of a boxed jewel is the sign of its contribution to $P_t$ (i.e., $1$ if the jewel is black, and $-1$ if it is white) and its \emph{degree} is the number of cycles it contains, if all the cycles are counterclockwise; it is minus this number if all the cycles are clockwise.
In other words, if $t$ is a tiling of a two-floored region whose sock is untangled and $\{b_i\}_i$ is the set of boxed jewels in this sock, then 
$$P_t(q) = a_0 + \sum_i \sgn(b_i) q^{\deg(b_i)},$$
where $a_0 = P_t(1) - \sum_i \sgn(b_i)$, $\sgn(b_i)$ is the sign of $b_i$ and $\deg(b_i)$, its degree. 

\begin{lemma}
\label{lemma:jewelGoesDown}
If $t_0$ and $t_1$ are tilings of a two-floored region such that their associated socks are both untangled (consist only of boxed jewels) and:
\begin{enumerate}[label=(\roman*)]
	\item The associated sock of $t_0$ consists of a single boxed jewel of degree $n > 0$.
	\item The associated sock of $t_1$ consists of $n$ boxed jewels of degree $1$ and same sign as the boxed jewel in $t_0$.
\end{enumerate} 
Then there exists a box with four floors $B$ where their embeddings lie in the same connected components. 
\end{lemma} 
\begin{proof}
\begin{figure}[ht]%
\centering
\includegraphics[width=0.4\columnwidth]{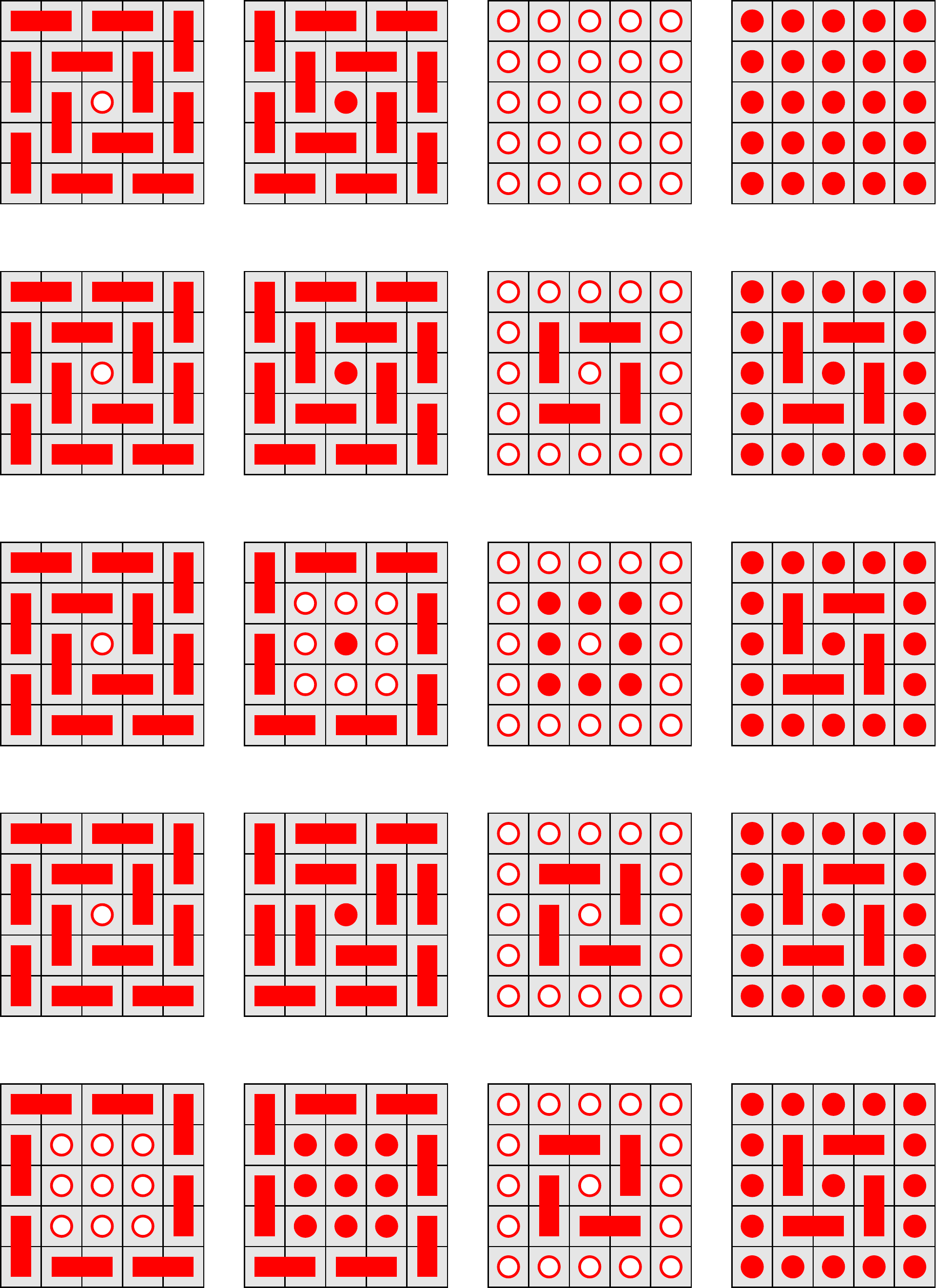}%
\caption{Five steps that bring a small boxed jewel from the top two floors to the bottom two floors, each consisting of four flips. In this case, a boxed jewel with degree $2$ in the top two floors is transformed into a boxed jewel with degree $1$ in the bottom two floors plus a cycle that can be easily flipped to a boxed jewel with degree $1$ in the top two floors. Since the bottom jewel can move freely in the bottom floors, it can be brought back up in a different position.}%
\label{fig:jewelGoesDown}%
\end{figure}

Figure \ref{fig:jewelGoesDown} illustrates the key step in the proof: that the innermost boxed jewel can be transported to the bottom two floors via flips. If the box $B$ is big enough, the innermost boxed jewel can freely move in the bottom two floors, and can eventually be brought back up outside of any other cycles. Since after this maneuver the bottom two floors are back as they originally were, the result of this maneuver is the embedding of a tiling whose sock is flip homotopic to an untangled sock with two boxed jewels, one with degree $n-1$ and another with degree $1$ (but both have the same sign as the original boxed jewel). Proceeding by induction and using Proposition \ref{prop:moreSpace}, we obtain the result.
\end{proof}

Therefore, boxed jewels with degree $n > 0$ (resp. degree $ -n < 0$) can be flipped into $n$ boxed jewels with degree $1$ (resp. $-1$). It only remains to see that boxed jewels with degrees $1$ and $-1$ and same sign cancel out.

\begin{lemma}
\label{lemma:jewelsCancel}
 If $t$ is a tiling of a duplex region whose sock is untangled and consists of two boxed jewels with degrees $1$ and $-1$ but same sign, then there exists a box with four floors where the embedding of $t$ is in the same flip connected component as the tiling consisting only of ``jewels'', that is, the tiling containing only $z$ dimers.
\end{lemma}
\begin{proof}
\begin{figure}[ht]%
\centering
\includegraphics[width=0.55\columnwidth]{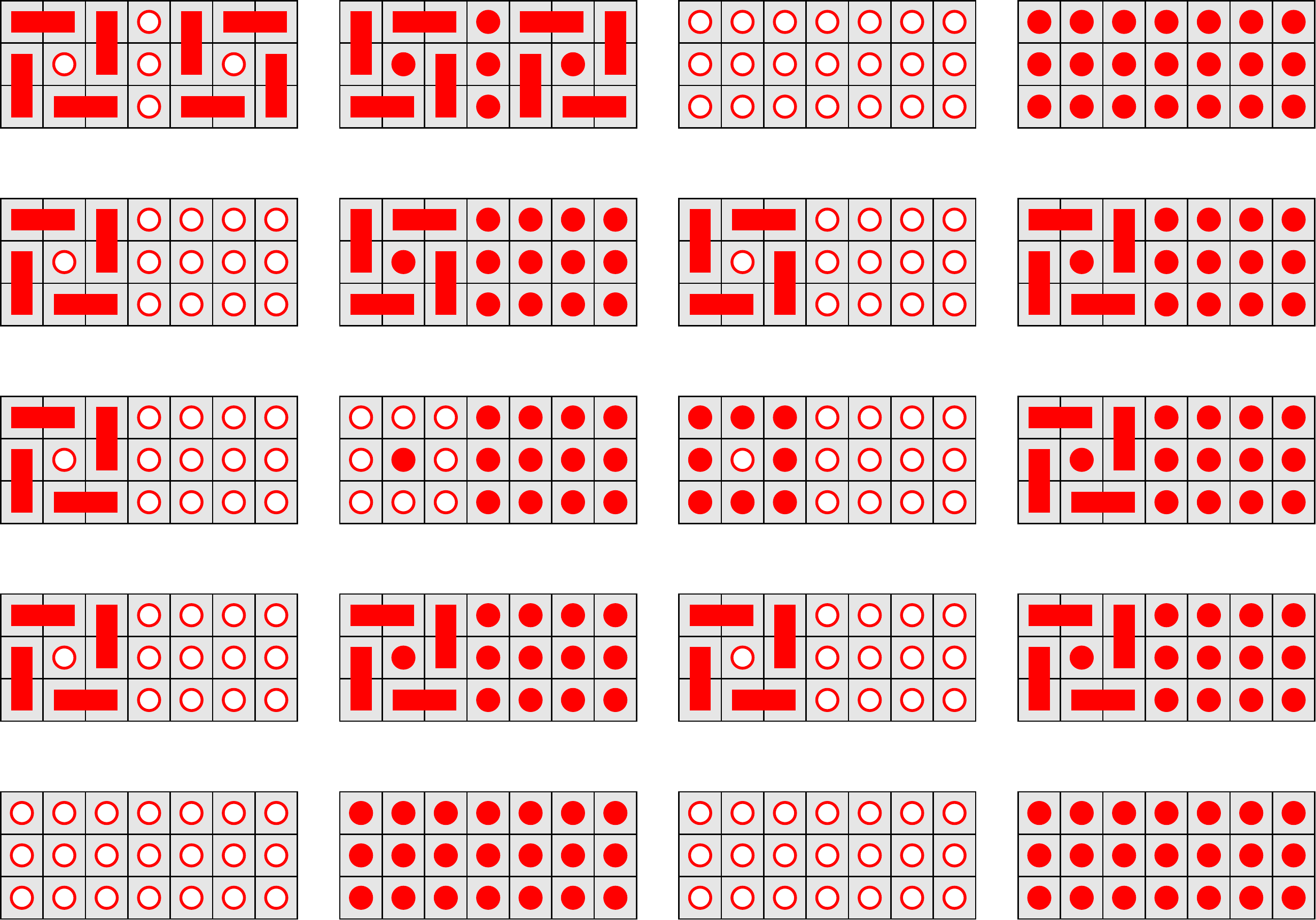}%
\caption{Illustration of how two boxed jewels with same sign but opposite degree cancel. One boxed jewel is transported to the bottom floor and there it moves so that it is exactly under the other boxed jewel (this is the first step). From then on, it is a relatively straightforward sequence of flips, and the three bottom drawings show some of the intermediate steps.}%
\label{fig:jewelsCancel}%
\end{figure}
The basic procedure is illustrated in Figure \ref{fig:jewelsCancel}. First, one jewel can be transported to the bottom floor using the procedure in Figure \ref{fig:jewelGoesDown}. There it can be moved so that it is right under the other boxed jewel, when they can easily be flipped into a dimer with only $z$ dimers.
\end{proof}

In short, what Lemmas \ref{lemma:jewelGoesDown} and \ref{lemma:jewelsCancel} imply is that a tiling $t$ whose sock is untangled and whose invariant is $P_t(q) = a_0 + \sum_{n \neq 0} a_n q^n$ can be embedded in a box with four floors in such a way that this embedding is in the same flip connected component as the embedding of another tiling $\tilde{t}$ (of a duplex region) whose sock is untangled and whose invariant is $P_{\tilde{t}}(q) = \tilde{a_0} + \left(\sum_{n \neq 0} na_n \right) q = \tilde{a_0} + P_t'(1)q$.

Using these last results, we can prove Proposition \ref{prop:embedFourFloors} as follows: let $t_0$ and $t_1$ be two tilings of a duplex region such that $P_{t_0}'(1) = P_{t_1}'(1)$. By Lemmas \ref{lemma:homotopicToUntangled} and \ref{lemma:equivEmbeddingHomotopy}, there exists a two-floored box $\tilde{B}$ where the embeddings of $t_0$ and $t_1$ are in the same connected component, respectively, as $\tilde{t_0}$ and $\tilde{t_1}$, two tilings whose sock is untangled. By Lemmas \ref{lemma:jewelGoesDown} and \ref{lemma:jewelsCancel} and the previous paragraph, there exists a box with four floors $B$ where the embeddings of $\tilde{t_0}$ and $\tilde{t_1}$ lie in the same connected component, respectively, as the embeddings of $\hat{t_0}$ and $\hat{t_1}$ (of the same duplex region), with invariants $P_{\hat{t_0}}(q) = a_0 + P_{t_0}'(1)q$ and $P_{\hat{t_1}}(q) = b_0 + P_{t_1}'(1) q$. Since $P_{t_0}'(1) = P_{t_1}'(1)$, it follows that $a_0 = b_0$ and so $P_{\hat{t_0}}(q) = P_{\hat{t_1}}(q)$. By Proposition \ref{prop:moreSpace}, the box $B$ can be chosen such that the embeddings of $\hat{t_0}$ and $\hat{t_1}$ lie in the same flip connected component; this concludes the proof.  

\section{Conclusion}
\label{sec:conclusion}
In this article, we looked at tilings of two-story regions by domino brick pieces, and our main concern were the properties of the flip. Flip invariance for tilings of more general regions, with an arbitrary number of floors, is discussed in \cite{segundoartigo}. Our results are summarized in the following theorems:

\begin{theo}[Properties of the invariant]
\label{theo:invariantProperties}
The polynomial $P_t(q)$, defined in sections \ref{sec:twoIdenticalFloors} and \ref{sec:generalTwoStory}, has the following properties:
\begin{enumerate}[label=(\roman*),topsep=0pc]
	\item \label{item:changeGhostCurves} If $R$ is a two-story region, then $P_t$ is uniquely defined for all tilings $t$ of $R$ up to a choice of ghost curves. Moreover, a different choice of ghost curves changes the polynomial $P_t$ for all tilings $t$ in the same way: by multiplying them by the same integer power of $q$.
	
	\item \label{item:flipInvariant} If $t_0$ and $t_1$ are tilings of a two-story region that lie in the same flip connected component, then $P_{t_0} = P_{t_1}$.
	
	\item \label{item:moreSpace} If $R$ is a duplex region and $t_0$ and $t_1$ are two tilings of $R$, then $P_{t_0} = P_{t_1}$ if and only if there exists a box with two floors $B$ such that the embeddings of $t_0$ and $t_1$ in $B$ lie in the same flip connected component.
	
	\item \label{item:fourFloursEmbedding} If $R, t_0, t_1$ are as above and $P_{t_0}'(1) = P_{t_1}'(1)$, then there exists a box with four floors $B$ such that the embeddings of $t_0$ and $t_1$ in $B$ lie in the same flip connected component.
	
\end{enumerate}

\end{theo}
\begin{proof}
For \ref{item:changeGhostCurves}, see Section \ref{sec:generalTwoStory}, and the discussion in the last three paragraphs in particular; 
\ref{item:flipInvariant} is Proposition \ref{prop:twoStoryFlipInvariant} (see also Proposition \ref{prop:twoFloorFlipInvariant});
\ref{item:moreSpace} is Proposition \ref{prop:moreSpace}; 
finally, \ref{item:fourFloursEmbedding} is Proposition \ref{prop:embedFourFloors}.
\end{proof}

\begin{rem}
The converse of \ref{item:fourFloursEmbedding} in Theorem \ref{theo:invariantProperties} also holds, that is, if there exists a box with four floors $B$ such that the embeddings of $t_0$ and $t_1$ in $B$ lie in the same flip connected component, then $P_{t_0}'(1) = P_{t_1}'(1)$. Indeed, we prove in \cite{segundoartigo} that the twist $\Tw(t)$, which equals $P_t'(1)$ in this case, is a flip invariant.
\end{rem}

\begin{theo}[Properties of the trit]
\label{theo:tritProperties}
The trit, which was introduced in Section \ref{sec:defsAndNotations}, has the following properties:
\begin{enumerate}[label=(\roman*)]
	\item \label{item:posTrit} If $t_0$ and $t_1$ are two tilings of a two-story region and $t_1$ can be reached from $t_0$ after a single positive trit, then $P_{t_1}(q) - P_{t_0}(q) = q^k(q-1)$ for some $k \in \ZZ$; as a consequence, $P_{t_1}'(1) - P_{t_0}'(1) = 1$.
	\item \label{item:connectivityFlipsTrits} The space of domino brick tilings of a duplex region is connected by flips and trits. In other words, if $t_0$ is any tiling of such a region and $t_1$ is any other tiling of the same region, then there exists a sequence of flips and trits that take $t_0$ to $t_1$.
\end{enumerate}
\end{theo}
\begin{proof}
\ref{item:posTrit} is Proposition \ref{prop:posTrit} and Corollary \ref{coro:tritstwofloors}; \ref{item:connectivityFlipsTrits} is Proposition \ref{prop:connectivityFlipsTrits}.
\end{proof}

\bibliography{biblio}{}
\bibliographystyle{plain}

\bigskip
\noindent
Departamento de Matem\'atica, PUC-Rio \\
Rua Marqu\^es de S\~ao Vicente, 225, Rio de Janeiro, RJ 22451-900, Brazil \\
\url{milet@mat.puc-rio.br}\\
\url{nicolau@mat.puc-rio.br}

\end{document}